\documentclass[12pt]{amsart}
	\nonstopmode
	\RequirePackage[colorlinks,citecolor=blue,urlcolor=blue,linkcolor=blue]{hyperref}
	\hypersetup{
		colorlinks = true,
		citecolor=blue,
		urlcolor=blue,
		linkcolor=blue,
		pdfpagemode = UseNone}
	\usepackage[titletoc,title]{appendix}
	\usepackage{graphicx,xspace,colortbl}
	\usepackage{amsmath,amsthm,amsfonts}
   	\usepackage{color}
	\usepackage{fancybox}
	\usepackage{epsfig}
	\usepackage{subfig}
	\usepackage{pdfsync}
    \def\qed{\hfill$\sqcap\kern-8.0pt\hbox{$\sqcup$}$\\}
    \def\re{\textnormal {Re}}
    \def\im{\textnormal {Im}}
    
    \def\e{{\mathbb E}}
    \def\r{{\mathbb R}}
    \def\c{{\mathbb C}}
    \def\d{{\textnormal d}}
    \def\i{{\textnormal i}}
    \def\G{\mathcal{G}}
    \def\L{\mathcal{L}}

    \def\sep{\hat{\sigma}}
	\newtheorem{theorem}{Theorem}
	\newtheorem{lemma}{Lemma}
	\newtheorem{proposition}{Proposition}
	
	\theoremstyle{definition}
	\newtheorem{definition}{Definition}

	\newtheorem{remark}{Remark}

	\DeclareMathOperator*{\sign}{sign}

   	\DeclareMathOperator*{\res}{res}
 
\title{Extending the Meijer  $G$-function} 

\author{ Dmitrii\:Karp}
\address{Department of Mathematics, Holon Institute of Technology, 52 Golomb Street, POB 305, Holon 5810201, Israel}
\email{dmitrika@hit.ac.il}

\author{Alexey \:Kuznetsov}
\address{Department of Mathematics and Statistics, York University,
Toronto, Ontario, M3J 1P3, Canada}
\email[Corresponding author]{akuznets@yorku.ca}
\thanks{Research of A.K. supported by the Natural Sciences and Engineering Research Council of Canada.}

\date{\today}

\begin{document}


\maketitle

\begin{abstract}
By replacing the Euler gamma function by the Barnes double gamma function in the definition of the Meijer $G$-function, we introduce a new family of special functions, which we call $K$-functions. This is a very general class of functions, which includes as special cases Meijer $G$-functions (thus also all hypergeometric functions ${}_p F_q$) as well as several new functions that appeared recently in the literature. Our goal is to define the $K$-function, study its analytic and transformation properties and relate it to several functions that appeared recently in the study of random processes and the fractional Laplacian. We further introduce  a generalization of the Kilbas-Saigo function and show that it is  a special case of $K$-function. 
\end{abstract}
{\vskip 0.25cm}
{\vskip 0.25cm}


\section{Introduction}\label{section:Intro}


Meijer $G$-function (see \cite{Mathai} and \cite{Jeffrey2007}[Chapter 9.3]) is defined as a Mellin-Barnes type integral 
\begin{equation}\label{def:Meijer_G}
G_{p,q}^{m,n}\Big( 
\begin{matrix}
a_1, \dots, a_p \\  b_1, \dots, b_q
\end{matrix} \Big
\vert z
\Big )=\frac{1}{2\pi \i} \int_{{\mathcal L}} 
\frac{\prod\limits_{j=1}^m \Gamma(b_j-s)\prod\limits_{j=1}^n \Gamma(1-a_j+s)}
	{\prod\limits_{j=m+1}^q \Gamma(1-b_j+s)\prod\limits_{j=n+1}^p \Gamma(a_j-s)} z^{s} \d s.
\end{equation}
Here $m,n,p,q$ are integers such that $0\le m \le q$ and $0\le n \le p$ and $a_j$, $b_j$ are complex numbers such that $a_k-b_j \notin {\mathbb N}$ for $1\le k \le n$ and $1\le j \le m$. The contour ${\mathcal L}$ runs either from $-\i \infty$ to $\i \infty$ or it is a loop beginning and ending at $+\infty$ (or $-\infty$).  In each case the contour ${\mathcal L}$ separates the points 
$$
\{b_j+l, \; 1\le j \le m, \; l\ge 0\},
$$
which are the poles of the functions $\Gamma(b_j-s)$, $j=1,2,\dots,m$
from the points  
$$
\{a_k-l, \; 1\le k \le n, \; l\ge 1\},
$$
which are the poles of $\Gamma(1-a_j+s)$, $j=1,2,\dots,n$. We refer the reader  to  \cite{Mathai} and \cite{Jeffrey2007}[Chapter 9.3] for complete description of the contour ${\mathcal L}$ and various properties of the Meijer $G$-function.  

The Meijer $G$-function unifies and generalizes most of the standard functions of mathematical physics. The family of $G$-functions is very useful and versatile  due to a number of stability properties: it is preserved by taking integrals and derivatives (of integer or fractional order), multiplication by powers, taking reciprocal of the variable, Laplace and Euler transforms and multiplicative convolution. This is the main reason why these functions play an important role in solving the hypergeometric differential equation (of order greater than $2$) \cite{KarpPrilepkina2016} and in fractional calculus \cite{Kiryakova2021}. Meijer's $G$-function is indispensable in probability and statistics: it represents the density of arbitrary products of beta and gamma distributed random variables \cite{CoelhoArnoldBook}, distribution of many likelihood ratio tests \cite{Mathai1984,TangGupta1986}, stationary distribution of certain Markov chains \cite{ChamayouLetac1999,Dufresne2010}.  Moreover, it has been observed that some universality laws in random matrix theory are represented by $G$-function kernels \cite{Kieburgetal2016}.   Some optical transfer functions have been recently shown to be expressible in terms of Meijer's $G$-function \cite{DiazMedina2017}.
When the Meijer $G$-function has a well-defined Mellin transform, multiplying or dividing that Mellin transform by gamma functions and taking the inverse Mellin transform often leads to another Meijer $G$-function. This is the main idea behind results in \cite{Dyda_2017}, where it was shown that the fractional Laplace operator maps certain Meijer $G$-functions into other Meijer $G$-functions.

In the last fifteen years a number of papers 
\cite{Jedidi2018,Kuznetsov2011,Kuznetsov2018,KuznetsovPardo2013,Ostrovsky} appeared in the  literature on probability theory, where certain random variables had densities given by functions defined by Mellin-Barnes integrals similar to \eqref{def:Meijer_G}, except that the gamma functions are replaced by double gamma functions $G(z;\tau)$ or $\Gamma_2(z;a,b)$. The first such example is the density function of the supremum of a stable process, it was studied in detail in  \cite{Hackmann2013,HubKuz2011,Kuznetsov2013} and was found to have very unusual series representation properties.  

Our main goal in this paper is to define an extension of the Meijer $G$-function that would cover all the examples that appeared so far in the literature and to study some basic properties of this new special function. Thus, in Section  \ref{section_Barnes_G} we introduce the Barnes double gamma function $G(z;\tau)$ and present its main properties that will be used later in the paper. In Section \ref{section_def_K_function} we define the extension of the Meijer $G$-function, which we call a $K$-function. In this section we also state analyticity properties of the $K$-function and several transformation identities. In Sections \ref{section_K_asymptotics}  we present results on the asymptotics of $K$-function when $z\to 0$ or $z\to \infty$. Section \ref{section_Mellin_K_function} is devoted to Mellin transform properties of this function and to deriving integral equations satisfied by the $K$-function. In section \ref{section_examples} we give examples of various functions that have already appeared in the literature and show how to express them in terms of $K$-function. In 
Appendix \ref{AppendixA} we  discuss an alternative definition of the $K$-function, where one would use
$\Gamma_2(z;a,b)$  (a symmetric version of the double gamma function) instead of Barnes' version $G(z;\tau)$  and we show that this would lead to essentially the same $K$-function, modulo a parameter change.

\section{Preliminaries: Barnes double gamma function}\label{section_Barnes_G}

Everywhere in this paper we will use the following notation.  We denote by 
    $$
    {\mathbb C}^+ := \{z \in {\mathbb C} \; : \; \im(z)>0 \}, \;\;\;
    {\mathbb C}^- := \{z \in {\mathbb C} \; : \; \im(z)<0 \}
    $$
    the upper and lower complex half-planes. We take the principal branch of the logarithm function  (with a branch cut along the negative half-line). For $z$ on the Riemann surface of the logarithm function we allow $\arg(z) \in (-\infty,\infty)$ and if $z\in {\mathbb C} \setminus (-\infty,0]$ then $\arg(z) \in (-\pi, \pi)$. 

The double gamma function $G(z;\tau)$ was introduced by Alekseevsky in 1888-1889 (see \cite{Neretin2024}) and a decade later it was studied extensively by Barnes \cite{Barnes1899,Barnes1901}. Now it is usually called Barnes (or Alekseevsky-Barnes) double gamma function and we will follow this convention in this paper. The function $G(z;\tau)$ is an entire function of $z$ of order two defined via the Weierstrass infinite product representation: 
\begin{equation}\label{Weierstrass_G}
G(z;\tau)=\frac{z}{\tau} e^{a(\tau)\frac{z}{\tau}+b(\tau)\frac{z^2}{2\tau}} \prod\limits_{m\ge 0} \prod\limits_{n\ge 0} {}^{'}
\left(1+\frac{z}{m\tau+n} \right)e^{-\frac{z}{m\tau+n}+\frac{z^2}{2(m\tau+n)^2}}, \;\;\; 
 |\arg(\tau)|<\pi, \;\;\; z\in \c, 
\end{equation}
where the prime in the second product means that the term corresponding to $m=n=0$ is omitted.
The functions $a(\tau)$ and $b(\tau)$ in \eqref{Weierstrass_G} can be expressed in terms of  {\it the gamma modular forms} $C(\tau)$ and $D(\tau)$
\begin{align*}
C(\tau)&:=\lim\limits_{m\to \infty} 
\bigg[ \sum\limits_{k=1}^{m-1}\psi(k\tau)+\frac12 \psi(m\tau)-\frac{1}{\tau}\ln\left(\frac{\Gamma(m\tau)}{\sqrt{2\pi}}\right) \bigg] , \\
D(\tau)&:=\lim\limits_{m\to +\infty} 
\bigg[ \sum\limits_{k=1}^{m-1}\psi'(k\tau)+\frac12 \psi'(m\tau)-\frac{1}{\tau}\psi(m\tau)\bigg],
\end{align*}
as follows
\begin{equation}\label{def_a_b}
a(\tau):=-\gamma \tau+\frac{\tau}2 \ln(2\pi \tau)+\frac12 \ln(\tau)-\tau C(\tau), \;\; 
b(\tau):=-\frac{\pi^2 \tau^2}{6} -\tau \ln(\tau)-\tau^2 D(\tau), 
\end{equation} where $\psi(z)=\Gamma'(z)/\Gamma(z)$ is the digamma function (see \cite{Jeffrey2007}) and $\gamma=-\psi(1)=0.577215...$ is the Euler-Mascheroni constant. 

Barnes also gave another infinite product representation for $G(z;\tau)$: 
\begin{equation}\label{eq_G2_inf_prod_Gamma}
G(z;\tau)=\frac{1}{\tau \Gamma(z)} e^{\tilde a(\tau) \frac{z}{\tau}+\tilde b(\tau)\frac{z^2}{2\tau^2}} 
\prod\limits_{m\ge 0} \frac{\Gamma(m\tau)}{\Gamma(z+m\tau)} e^{z\psi(m\tau)+\frac{z^2}2 \psi'(m\tau)},
\end{equation}
where the functions $\tilde a(\tau)$ and $\tilde b(\tau)$ are given by
\begin{equation}\label{def_tilde_a_tilde_b}
\tilde a(\tau)=a(\tau)-\gamma \tau, \qquad
\tilde b(\tau)=b(\tau)+\frac{\pi^2 \tau^2}{6}.
\end{equation}

The double gamma function has a number of useful properties. First of all, we have a normalization condition $G(1;\tau)=1$. Second, the double gamma function is quasi-periodic with periods $1$ and $\tau$
\begin{equation}\label{eq:funct_rel_G}
G(z+1;\tau)=\Gamma\left(\frac{z}{\tau}\right) G(z;\tau), \;\;\; G(z+\tau;\tau)=(2\pi)^{\frac{\tau-1}2}\tau^{-z+\frac12}
\Gamma(z) G(z;\tau).
\end{equation} 
It is clear from the Weierstrass product \eqref{Weierstrass_G} that the function $G(z;\tau)$  has zeros on the lattice $m\tau+n$, $m\le 0$, $n\le 0$ and that the zeros are all simple if and only if $\tau$ is not a rational number.  
The double gamma function satisfies other important identities: 
\begin{itemize}
\item[(i)] Modular transformation 
\begin{equation}\label{eq:G_1_over_tau}
G(z;\tau)=(2\pi)^{\frac{z}2 \left(1-\frac1{\tau} \right)} \tau^{\frac{z-z^2}{2\tau}+\frac{z}2-1}  G\left(\frac{z}{\tau};\frac{1}{\tau}\right).
\end{equation}
\item[(ii)] A multiplication formula: for any integers $p,q \in {\mathbb N}$
 \begin{equation}\label{G_z_p/q_formula2}
 G(z;p \tau/q )=
 q^{\frac{1}{2 p \tau}(z-1)(qz-p \tau) }  (2\pi)^{-\frac{q-1}{2}(z-1)}
 \prod\limits_{i=0}^{p-1} \prod\limits_{j=0}^{q-1} 
 \frac{G((z+i)/p+j \tau/q;\tau)}{G((1+i)/p+j \tau/q;\tau)}.
 \end{equation} 
 \item[(iii)]  A product identity 
  \begin{equation}\label{eq:modular}
  G(z;\tau)=\Big(\frac{1+\tau}{\tau}\Big)^{\frac{z^2}{2\tau}-(1+\tau) \frac{z}{2\tau}+1}
  (2\pi)^{-\frac{z}{2\tau}}G(z+1;1+\tau)G(z/\tau;1+1/\tau).
 \end{equation}
\end{itemize}

The proofs of \eqref{eq:G_1_over_tau} and 
\eqref{G_z_p/q_formula2} can be found in  \cite{Barnes1901}. For the proof of (iii) see \cite{AK_2023}.

The double gamma function $G(z;\tau)$ is closely related to the symmetric version of this function, denoted by $\Gamma_2(z;\omega_1,\omega_2)$, see \cite{Barnes1901} (note that this function is denoted by $\Gamma^B$ in \cite{Ruijsenaars2000} and some other papers). We consider this function only for $\omega_1,\omega_2>0$. It is known  that $\Gamma_2(z;\omega_1,\omega_2)$ is symmetric in $\omega_1$ and $\omega_2$ and 
 satisfies the functional identities
\begin{equation}\label{Gamma_2_functional_equations}
\Gamma_2(z+\omega_1;\omega_1,\omega_2)=
\sqrt{2\pi}\frac{\omega_2^{\frac{1}{2}-\frac{z}{\omega_2}}}{\Gamma(
\frac{z}{\omega_2})}
\Gamma_2(z;\omega_1,\omega_2),  \;\;\;
\Gamma_2(z+\omega_2;\omega_1,\omega_2)=
\sqrt{2\pi}\frac{\omega_1^{\frac{1}{2}-\frac{z}{\omega_1}}}{\Gamma(
\frac{z}{\omega_1})}
\Gamma_2(z;\omega_1,\omega_2),
\end{equation}
and the normalization condition 
$\Gamma_2(\omega_1;\omega_1,\omega_2)=\sqrt{2\pi /\omega_2}$. As Barnes established in \textsection 26 in
\cite{Barnes1901}, the function $\Gamma_2(z;\omega_1,\omega_2)$ can be expressed in terms of $G(z;\tau)$ as follows: 
\begin{equation}\label{G_as_Gamma1}
    \Gamma_2(z;\omega_1,\omega_2)=
    (2\pi)^{\frac{z}{2\omega_1}}
    \omega_2^{-\frac{z^2}{2\omega_1\omega_2}+\frac{z(\omega_1+\omega_2)}{2\omega_1\omega_2}-1}
    G\Big(\frac{z}{\omega_1}; \frac{\omega_2}{\omega_1} \Big)^{-1}.
\end{equation}

There exists a Binet type integral representation for the double gamma function, which we state in the next proposition. This representation  is implicitly contained in the results of Barnes 
\cite{Barnes1901}, however, we were not able to locate this formula in this exact form in the existing literature and thus we provide a proof here.

\begin{proposition}
For $\re(z)>0$ and $\re(\tau)>0$ we have
   \begin{align}
 \label{lnG_integral2}
\ln G(z;\tau)&=(a_2(\tau) z^2+a_1(\tau) z + a_0(\tau)) \ln(z)+b_2(\tau) z^2+ b_1(\tau) z + b_0(\tau)
- \int_0^{\infty} e^{-zx} f_3(x) \d x,
 \end{align}
 where 
 $$
 f_3(x):=\frac{1}{x^3} \Big( f(x) - f(0) -f'(0)x - 
 \frac{1}{2} f''(0) x^2 \Big), \;
 {\textnormal{ and }} \;
 f(x):=\frac{x^2}{(1-e^{-x})(1-e^{-\tau x})}. 
 $$
 In the above formula, we have 
  $a_1(\tau)=-(1+\tau)/(2\tau)$, $a_2(\tau)=1/(2\tau)$, $b_2(\tau)=-(3/2+\ln(\tau) )/(2\tau)$ and the explicit expressions for other coefficients can be found in \cite{AK_2023,Bill1997}.
\end{proposition}
\begin{proof}
Assume first that $\tau>0$. Formulas (3.13), (3.18), (3.19) in \cite{Ruijsenaars2000} with $N=M=2$ and $a_1=1$, $a_2=\tau$  imply that (note that our $\Gamma_2$ is $\Gamma_2^B$ in \cite{Ruijsenaars2000}, up to a multiplicative constant)
for $\re(z)>0$
$$
\ln \Gamma_2(z;1,\tau)=\tilde P(z;\tau) \ln(z)+ \tilde Q(z;\tau)+ \int_0^{\infty} e^{-zx} f_3(x) \d x,
$$
where $\tilde P(z;\tau)$ and $\tilde Q(z;\tau)$ are polynomials of degree two in $z$-variable (with coefficients depending on $\tau$).  Using the above result and formula \eqref{G_as_Gamma1}  we conclude that for $\re(z)>0$
\begin{equation}\label{lnG_integral2b}
\ln G(z;\tau)= P(z;\tau) \ln(z)+ Q(z;\tau) - \int_0^{\infty} e^{-zx} f_3(x) \d x,
\end{equation}
for some polynomials 
\begin{align*}
     P(z;\tau)&=a_2(\tau) z^2 + a_1(\tau) z + a_0(\tau), \\
     Q(z;\tau)&=b_2(\tau) z^2 + b_1(\tau) z + b_0(\tau).
\end{align*}
Using the fact that $f_3(x)$ is analytic in the neighbourhood of $x=0$ and applying Watson's Lemma, we conclude that 
$$
 \int_0^{\infty} e^{-zx} f_3(x) \d x=O(1/z)
$$
as $z\to +\infty$ along positive half-line $z\in (0,\infty)$. Thus formula \eqref{lnG_integral2b} gives an asymptotic representation of $\ln G(z;\tau)$ as $z\to +\infty$, and comparing the coefficients of this asymptotic expansion with those obtained by Billingham and King in \cite{Bill1997}, we find the values of the coefficients $a_1(\tau)$, $a_2(\tau)$ and $b_2(\tau)$. This ends the proof of \eqref{lnG_integral2} for $\tau>0$ and the general case when $\re(\tau)>0$ follows by analytic continuation. 
\end{proof}

We will also require the following asymptotic expansion of $\ln G(z;\tau)$. 
\begin{proposition}\label{prop_asymptotics_G}
For any fixed $\tau>0$, we have
\begin{equation}\label{eqn_lnG_E}
\ln G(z;\tau)=\Big(\frac{1}{2\tau} z^2- \frac{1+\tau}{2\tau} z +a_0(\tau)\Big) \ln(z)-\frac{1}{2\tau} \Big( \frac{3}{2}+\ln(\tau) \Big) z^2+ b_1(\tau) z +b_0(\tau)
+ O(|z|^{-1}),
\end{equation}
as $z \to \infty$, uniformly in any sector $|\arg(z)|<\pi-\epsilon<\pi$.
\end{proposition}
This asymptotic result follows from the asymptotic formula for $\Gamma_2(z;a,b)$, derived by Barnes in \cite{Barnes1901}. 
This asymptotic relation was rediscovered by Billingham and King in \cite{Bill1997} (see also \cite{AK_2023}, where explicit formula for all terms in the complete asymptotic expansion is given). Neither Barnes nor Billingham and King state explicitly that the asymptotic relation \eqref{eqn_lnG_E} holds uniformly in the sector $|\arg(z)|<\pi-\epsilon<\pi$, though this fact does follow implicitly from Barnes' proof of \eqref{eqn_lnG_E}. The proof that \eqref{eqn_lnG_E} is uniform in the given sector can be obtained fairly quickly from formula \eqref{lnG_integral2}. We will only sketch the main steps here. We fix $\theta \in (-\pi/2,\pi/2)$ and rotate the contour of integration in the integral on the right-hand side of \eqref{lnG_integral2}, so that after a change of variable $x=e^{\i \theta} y$ we obtain 
\begin{equation}\label{f3_integral}
\int_0^{\infty} e^{-zx} f_3(x) \d x=
e^{\i \theta} \int_0^{\infty} e^{-e^{\i \theta} zy} f_3(e^{\i \theta} y) \d y.
\end{equation}
The rotation of the contour of integration is justified, since for $\tau>0$ the function $f_3(x)$ is analytic and bounded in any sector $|\arg(x)|<\pi/2-\epsilon$. 
It is clear that the integral on the right-hand side of 
\eqref{f3_integral} is $O(1/z)$ as $z\to \infty$, uniformly in any sector $|\theta+\arg(z)|<\pi/2-\epsilon$. Since $\theta$ can be an arbitrary number in the interval $(-\pi/2,\pi/2)$, we obtain that the desired result \eqref{eqn_lnG_E} is uniform in $|\arg(z)|<\pi-\epsilon$.

\section{Definition of the Meijer-Barnes $K$ function}\label{section_def_K_function}

We assume that $m,n,p,q$ are non-negative integers and denote ${\bf a}=(a_1,\dots, a_p) \in {\mathbb C}^p$ and 
${\bf b}=(b_1,\dots, b_q) \in {\mathbb C}^q$. 
If ${\bf a} \in {\mathbb C}^p$ is a vector and $c$ is a complex number, we will denote by $c  {\bf a}$ and $c+{\bf a}$ vectors of length $p$ having elements $ca_i$ and $c+a_i$, respectively. 
For $s\in {\mathbb C}$ and $\tau \in \c \setminus (-\infty,0]$  we define
\begin{equation}\label{eq:integrand}
\G_{p,q}^{m,n}(s)=\G_{p,q}^{m,n}\Big(\begin{matrix} {\bf a} \\
{\bf b} \end{matrix}
\Big\vert s; \tau\Big):=\frac{\prod\limits_{j=1}^m G(b_j-s;\tau)\prod\limits_{j=1}^n G(1+\tau-a_j+s;\tau)}
	{\prod\limits_{j=m+1}^q G(1+\tau-b_j+s;\tau)\prod\limits_{j=n+1}^p G(a_j-s;\tau)}.     
\end{equation}
This function is an analogue of the product/ratio of gamma functions in \eqref{def:Meijer_G}. We define the two sets 
\begin{align}\label{def:Lambda_pm}
\Lambda^+:&=\{a_j+l \tau + k \; : \; n+1\le j \le p, \; l \ge 0, \; k\ge 0\} , \\
\nonumber
\Lambda^-:&=\{b_j-l \tau - k \; : \; m+1\le j \le q, \; l \ge 1, \; k\ge 1\}. 
\end{align}
From the properties of the double gamma function, it is clear that $\G_{p,q}^{m,n}(s)$ is a meromorphic function having poles precisely at the points in $\Lambda^+ \cup \Lambda^-$.

Next we introduce the parameters: 
\begin{equation}\label{eq:parameters}
\begin{split}
N&:=2(m+n)-p-q, \\ 
\nu&:=\sum_{j=1}^p a_j-\sum_{j=1}^q b_j, \\ 
\mu&:=\sum\limits_{j=1}^n a_j - \sum\limits_{j=n+1}^p a_j+
\sum\limits_{j=1}^m b_j-\sum\limits_{j=m+1}^q  b_j,\\ 
\xi&:=\sum_{j=1}^n a_j^2-\sum_{j=n+1}^p a_j^2+\sum_{j=1}^m b_j^2-
\sum_{j=m+1}^q b_j^2.
\end{split}
\end{equation}
In the following proposition we collect several transformation properties of $\G_{p,q}^{m,n}(s)$. 

\begin{proposition}
For $s\in {\mathbb C}$, $\tau \in \c \setminus (-\infty,0]$ and $c\in {\mathbb C}$ the following identities hold: 
\begin{equation}\label{eq:G_identity1}
\G_{q,p}^{n,m}\Big(\begin{matrix} 1+\tau-{\bf b} \\
1+\tau-{\bf a} \end{matrix}
\Big\vert -s ;\tau\Big)=\G_{p,q}^{m,n}\Big(\begin{matrix} {\bf a} \\
{\bf b} \end{matrix}
\Big\vert s ; \tau\Big)=\G_{p,q}^{m,n}\Big(\begin{matrix} c+ {\bf a} \\
c+{\bf b} \end{matrix}
\Big\vert c+s; \tau\Big),
\end{equation}
and
\begin{equation}\label{eq:G_s_over_tau_transformation}
\G_{p,q}^{m,n}\Big(\begin{matrix} {\bf a} \\
{\bf b} \end{matrix}
\Big\vert s; \tau\Big)
=A \times B^{-s} \times \tau^{-\frac{N s^2}{2\tau}} \times \G_{p,q}^{m,n}\Big(\begin{matrix} \tau^{-1} {\bf a} \\
\tau^{-1} {\bf b} \end{matrix}
\Big\vert s\tau^{-1}; \tau^{-1}\Big),
\end{equation}
where 
\begin{equation}\label{def_A_B}
A:=(2\pi)^{\frac{1-\tau}{2\tau}(\nu-(1+\tau)(n+m-q))}
\tau^{\frac{1}{2\tau} ((1+\tau)\mu-\xi)-N}, \;\;\; 
B:=(2\pi)^{\frac{1-\tau}{2\tau}(p-q)} \tau^{\frac{1}{2\tau}(N(1+\tau)-2\mu)}.
\end{equation}
\end{proposition}
\begin{proof}
    The two identities in \eqref{eq:G_identity1} follow immediately from
    \eqref{eq:integrand}. The proof of \eqref{eq:G_s_over_tau_transformation} follows from modular transformation \eqref{eq:G_1_over_tau}
    and can be found in Appendix \ref{AppendixB}.
\end{proof}

For $\epsilon \in (0,\pi)$ we introduce two sectors
$$
{\mathcal S}^{\pm}_{\epsilon}: = \{ z\in {\mathbb C} \; : \; z\neq 0, 
\; \epsilon \le \pm \arg(z) \le \pi - \epsilon \}. 
$$
The next proposition states the asymptotic behavior of $\G_{p,q}^{m,n}\Big(\begin{matrix} {\bf a} \\
{\bf b} \end{matrix}
\Big\vert s ; \tau\Big)$ as $s\to \infty$ in sectors ${\mathcal S}^{\pm}_{\epsilon}$. 

\begin{proposition}\label{prop:phi-asymp}
Let $\epsilon \in (0,\pi)$ and $\tau>0$. Denote $\chi=\sign (\im(s))$. Then as $s\to \infty$ in sectors ${\mathcal S}^{\pm}_{\epsilon}$ we have
\begin{itemize}
\item[(i)]
if $N\neq 0$, then 
\begin{equation}\label{eq:phi_asymptotics_case1}
\G_{p,q}^{m,n}\Big(\begin{matrix} {\bf a} \\
{\bf b} \end{matrix}
\Big\vert s ; \tau\Big)=\exp\Big(\frac{N}{2\tau} s^2 \ln(s)+O(|s|^2)\Big);
\end{equation}
\item[(ii)]
if $N=0$,  we have 
\begin{equation}\label{eq:phi_asymptotics_case2}
\G_{p,q}^{m,n}\Big(\begin{matrix} {\bf a} \\
{\bf b} \end{matrix}
\Big\vert s ; \tau\Big)=
\exp\Big(\frac{\pi \i \chi}{ 4\tau} (p-q) s^2 -\frac{1}{\tau}\mu s\ln(s) +O(|s|)\Big);
\end{equation}
\item[(iii)]
if $N=p-q=0$,  we have 
\begin{equation}\label{eq:phi_asymptotics_case3}
\G_{p,q}^{m,n}\Big(\begin{matrix} {\bf a} \\
{\bf b} \end{matrix}
\Big\vert s ; \tau\Big)=
\exp\Big(-\frac{1}{\tau}\mu s\ln(s)+ 
\frac{1}{\tau}(1+\ln(\tau)) \mu s  + \frac{\pi \i \chi}{2 \tau} (\mu-\nu)  s  + O(\ln |s|)\Big);
\end{equation}
\item[(iv)] If $N=p-q=\mu=\nu=0$, then 
\begin{equation}\label{eq:phi_asymptotics_case4}
\G_{p,q}^{m,n}\Big(\begin{matrix} {\bf a} \\
{\bf b} \end{matrix}
\Big\vert s ; \tau\Big)=\exp\Big(\frac{1}{2\tau} \xi \ln |s| +O(1)\Big).
\end{equation}
\end{itemize}
All these asymptotic results hold uniformly as $s \to \infty$ in the corresponding sectors. 
\end{proposition}
\begin{proof}
We denote the coefficients of the asymptotic expansion in \eqref{eqn_lnG_E} by $A_1=1/(2\tau)$, $A_2=-(3/2+\ln(\tau))/(2\tau)$, $A_3=-(1+\tau)/(2\tau)$, $A_4=a_0(\tau)$ and $A_5=b_1(\tau)$. Using the asymptotic expansion \eqref{eqn_lnG_E} and the following result
$$
\ln(c+s)=\ln(s)+c/s+O(s^{-2}), \;\;\; s\to \infty, 
$$
we find that 
\begin{align}
\label{eq:G(c+z)_asymptotics}
\ln (G(c+s;\tau))&=A_1s^2\ln(s)+A_2s^2+(A_3+2cA_1)s\ln(s)\\
\nonumber
&+(A_5+c(A_1+2A_2))s+ (A_1c^2+A_3c+A_4)\ln (s)+O(1), 
\end{align}
as $s\to \infty$ staying in ${\mathcal S}^+_{\epsilon}$. We recall that we consider the principal branch of the logarithm, for which we have
$$
\ln(-s)=\ln(s)-\pi \i, \;\;\; s\in {\mathbb C}^+,
$$
From this formula and from \eqref{eq:G(c+z)_asymptotics} we obtain
\begin{align}
\label{eq:G(c-z)_asymptotics}
\ln (G(c-s;\tau))&=A_1s^2\ln(s)+(A_2-A_1 \pi \i)s^2-(A_3+2cA_1)s\ln(s)\\
\nonumber
&-(A_5+c(A_1+2A_2)-(A_3+2cA_1) \pi \i)s+(A_1c^2+A_3c+A_4)\ln (s)+O(1),  
\end{align}
as $s\to \infty$ in the sector  ${\mathcal S}^{+}_{\epsilon}$.

Now we are ready to prove all the statements of Proposition \ref{prop:phi-asymp}. We will consider the case when $s\to \infty$ in ${\mathcal S}^+_{\epsilon}$, so that $\chi=1$ (the other case $s\in {\mathcal S}^-_{\epsilon}$ can be obtained in exactly the same way). 
First, we note that the function $\G_{p,q}^{m,n}(s)$  in \eqref{eq:integrand} has $m+n$ factors $G(\cdot;\tau)$ in the numerator and $p+q-m-n$ such factors in the denominator. The logarithm of each factor has a dominant asymptotic term $s^2 \ln(s)/(2\tau)$ and all remaining terms are bounded by $O(|s|^2)$, and this implies the result in  \eqref{eq:phi_asymptotics_case1}. 

Assume now that $N=0$ (which implies $m+n=(p+q)/2$). Inspecting asymptotic formulas \eqref{eq:G(c+z)_asymptotics} and \eqref{eq:G(c-z)_asymptotics} we see that, in the asymptotic expansion of $\ln \G_{p,q}^{m,n}(s)$ as $s\to \infty$ in $s\in {\mathcal S}^{+}_{\epsilon}$, the terms 
$$
A_1s^2\ln(s), \; A_2s^2, A_3s\ln(s)
$$
are canceled (as the number of such terms in the numerator of \eqref{eq:integrand} and in the denominator is the same). The term $-A_1\pi \i s^2$ in \eqref{eq:G(c-z)_asymptotics} contributes 
$$
-A_1 \pi \i s^2 (m-(p-n))=\frac{\pi \i}{4 \tau} (p-q) s^2.
$$
The terms $2cA_1s \ln(s)$ in \eqref{eq:G(c+z)_asymptotics} and 
$-2cA_1 s\ln(s)$ in \eqref{eq:G(c-z)_asymptotics} contribute
$$
2A_1 \bigg[-\sum\limits_{j=1}^m b_j + \sum\limits_{j=1}^n (1+\tau-a_j)- \sum\limits_{j=m+1}^q (1+\tau-b_j)+ \sum\limits_{j=n+1}^p a_j \bigg]s \ln(s)=\frac{1}{\tau}\Big(\frac{p-q}{2}(1+\tau)-\mu\Big)s\ln(s)
$$
and the terms $A_3s \ln(s)$ in \eqref{eq:G(c+z)_asymptotics} and 
$-A_3 s\ln(s)$ in \eqref{eq:G(c-z)_asymptotics} contribute
$$
A_3(n-m-(q-m)+(p-n))s
\ln(s)=-\frac{1+\tau}{2\tau}(p-q)s\ln(s).
$$
Combining the above three terms we obtain the result in \eqref{eq:phi_asymptotics_case2}. 

Consider now the case when $N=p-q=0$. Note that these conditions imply that $m+n=p=q$. The terms $A_5s$ in \eqref{eq:G(c+z)_asymptotics} and 
$-A_5s$ in \eqref{eq:G(c-z)_asymptotics} contribute
$$
A_5\times (-m+n-(q-m)+(p-n))s=0.
$$
The terms $c(A_1+2A_2)s$ in \eqref{eq:G(c+z)_asymptotics} and 
$-c(A_1+2A_2)s$ in \eqref{eq:G(c-z)_asymptotics} contribute
$$
(A_1+2A_2)\times \Big( (1+\tau)(n+m-q)-\mu \Big) s=-(A_1+2A_2) \mu s=\frac{1}{\tau}\Big( 1+\ln(\tau)\Big) \mu s.
$$
The terms $A_3 \pi \i s$ in \eqref{eq:G(c-z)_asymptotics} contribute
$$
- A_3 \pi \i (m-(p-n))s=0. 
$$
The terms $2c A_1 \pi \i s$ in \eqref{eq:G(c-z)_asymptotics} contribute 
$$
2 A_1 \pi \i \Big( \sum\limits_{j=1}^m b_j-\sum\limits_{j=n+1}^p a_j \Big)s=\frac{\pi \i}{2\tau}(\mu-\nu) s,
$$
and this ends the proof of \eqref{eq:phi_asymptotics_case3}. 
The proof of \eqref{eq:phi_asymptotics_case4} follows the same steps and will be omitted.
\end{proof}

Next we fix $\alpha \in \c$  and denote
\begin{equation}\label{def_phi}
\phi(s):=\G_{p,q}^{m,n}\Big(\begin{matrix} {\bf a} \\
{\bf b} \end{matrix}
\Big\vert s; \tau\Big)e^{\frac{\pi \alpha}{\tau}{s^2}}. 
\end{equation}
Our plan is to define the $K$-function as a Mellin-Barnes integral 
\begin{equation}\label{eq:def_K_function0}
K_{p,q}^{m,n}\Big( 
\begin{matrix}
{\bf a} \\  {\bf b} 
\end{matrix} \Big
\vert z; \tau, \alpha
\Big ):=\frac{1}{2\pi \i} \int\limits_{\gamma} 
\phi(s) z^{-s}ds.
\end{equation}
To achieve this we need to decide how to choose an appropriate contour of integration $\gamma$. We assume from now on that $\Lambda^+ \cap \Lambda^- = \emptyset$ (we recall that these sets are defined in \eqref{def:Lambda_pm}). Under this assumption,  there exists a contour $\gamma$ which divides the complex plane into two domains such that the set $\Lambda^+$ lies in the domain containing large positive numbers and $\Lambda^-$ lies in the other domain.  The direction of the contour will be chosen to leave the set $\Lambda^-$ on the left.  We will also assume that for $s$ large enough and $\im(s)>0$ the contour $\gamma$  coincides with a half-line with the slope $\tan(\theta^+)$ (for some $\theta^+ \in (0,\pi)$) and
for $s$ large enough and $\im(s)<0$ it coincides with a half-line with the slope  $\tan(\theta^-)$ (for some $\theta^- \in (-\pi, 0)$). We will denote such contours of integration by ${\mathcal L}_{\theta^-,\theta^+}$.  The contour
$\L_{-\pi/2,\pi/2}$  will be denoted by $\L_{\i\infty}$.  Examples of two such contours ${\mathcal L}_{-\pi/4,\pi-\epsilon}$ and $\L_{\i\infty}$ can be seen on Figure \ref{fig1}. The picture in Figure \ref{fig1}  with the contour $\L_{\i \infty}$ shows how to choose the contour when $b_j-a_k> 0$ for some $j$ and $k$ satisfying $m+1 \le j \le q$ and $n+1\le k \le p$. 

\begin{figure}[t!]
\centering
\includegraphics[height =6.5cm]{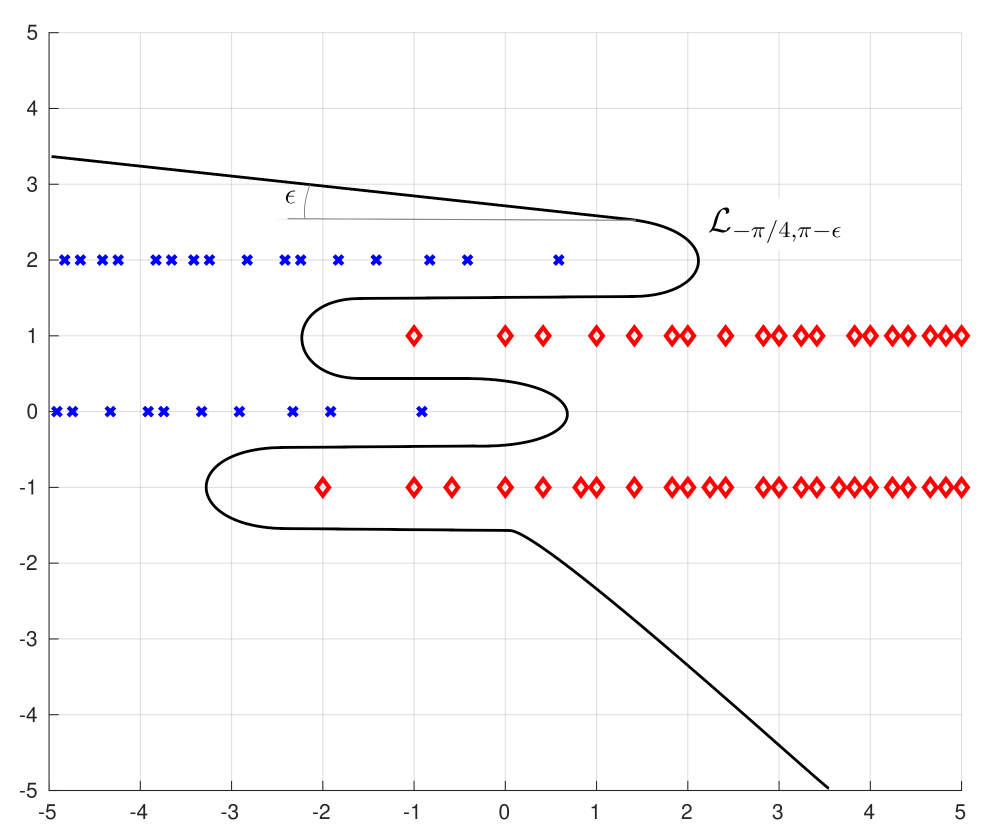}
\includegraphics[height =6.5cm]{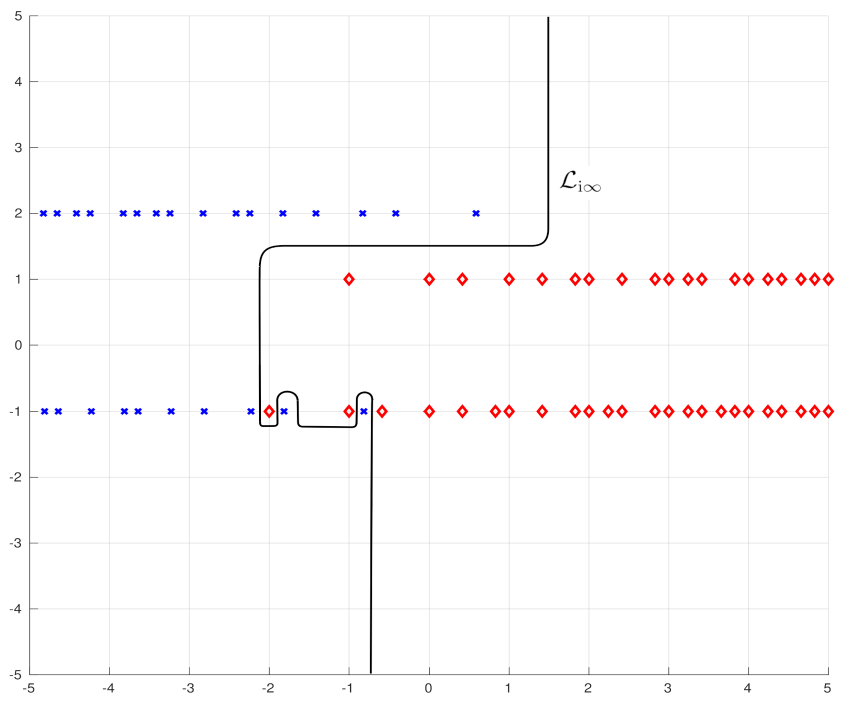}
\caption{Examples of contours ${\mathcal L}_{\theta^+,\theta^-}$, separating the sets $\Lambda^+$ (red diamonds) and $\Lambda^-$ (blue crosses) and having prescribed angles at infinity.} 
\label{fig1}
\end{figure}

The angles $\theta^{\pm}$ have to be chosen carefully, so that the Mellin-Barnes integral \eqref{eq:def_K_function0} over the contour $\gamma=\L_{\theta^-,\theta^+}$ converges. We will distinguish seven different cases of parameters, depending on the leading asymptotic term of the function $\ln \phi(s)$ in the sectors  ${\mathcal S}^{\pm}_{\epsilon}$. According to Proposition \ref{prop:phi-asymp}, this leading term can be one of $s^2 \ln(s), s^2, s\ln(s), s$ or $ \ln(s)$.  In cases (i)-(v) we will have leading asymptotic terms of the form $s^2 \ln(s), s^2$ or $s\ln(s)$ and we will choose any $\theta^+ \in (0,\pi)$ (respectively, $\theta^- \in (-\pi,0)$ such that the real part of the leading asymptotic term of $\ln \phi(s)$ goes to 
$-\infty$ as $s\to \infty$ along the ray $\arg(s)=\theta^+$ (respectively, $\arg(s)=\theta^-$).  In the last two cases (vi) and (vii) the leading asymptotic terms can be $s$ or $\ln(s)$ and these cases will be treated separately. We will also ensure, by restricting the parameters, that 
\begin{itemize}
\item[{\bf (a)}] {\bf  it is possible to choose such $\theta^{\pm}$}, 
\item[{\bf (b)}] {\bf any two such choices will necessarily lead to the same value of the integral.} \label{conditions_a_b}
\end{itemize}

To better understand the discussion that follows, it is helpful to recall some basic geometric facts about how linear and quadratic functions map the complex plane (see Figure \ref{fig2}). In particular, these graphs show the sectors in the complex plane (white regions) where the real parts of $\omega s$ and $\omega s^2$ are negative.

\begin{figure}[t!]
\centering
\includegraphics[height =6.5cm]{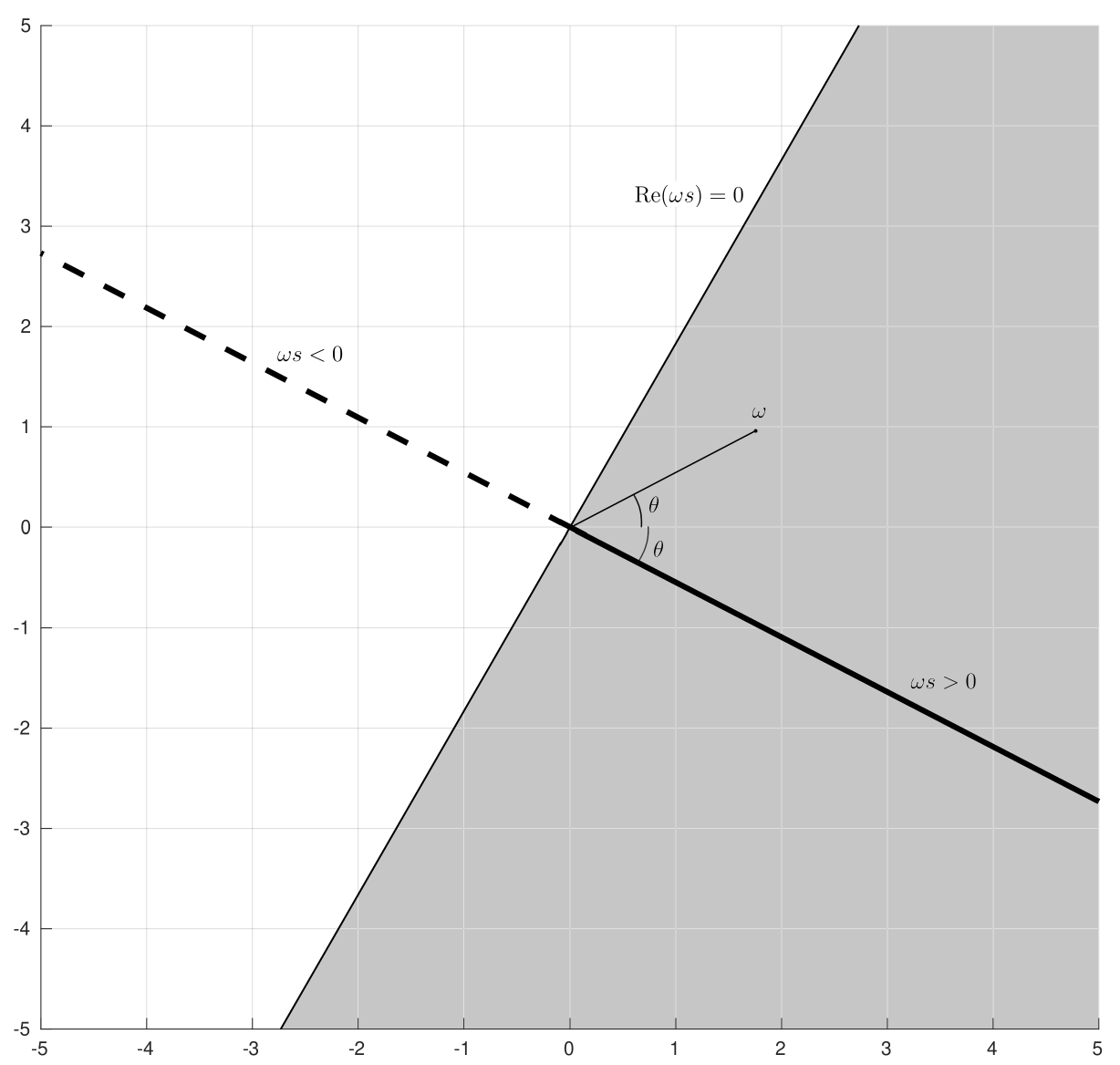}
\includegraphics[height =6.5cm]{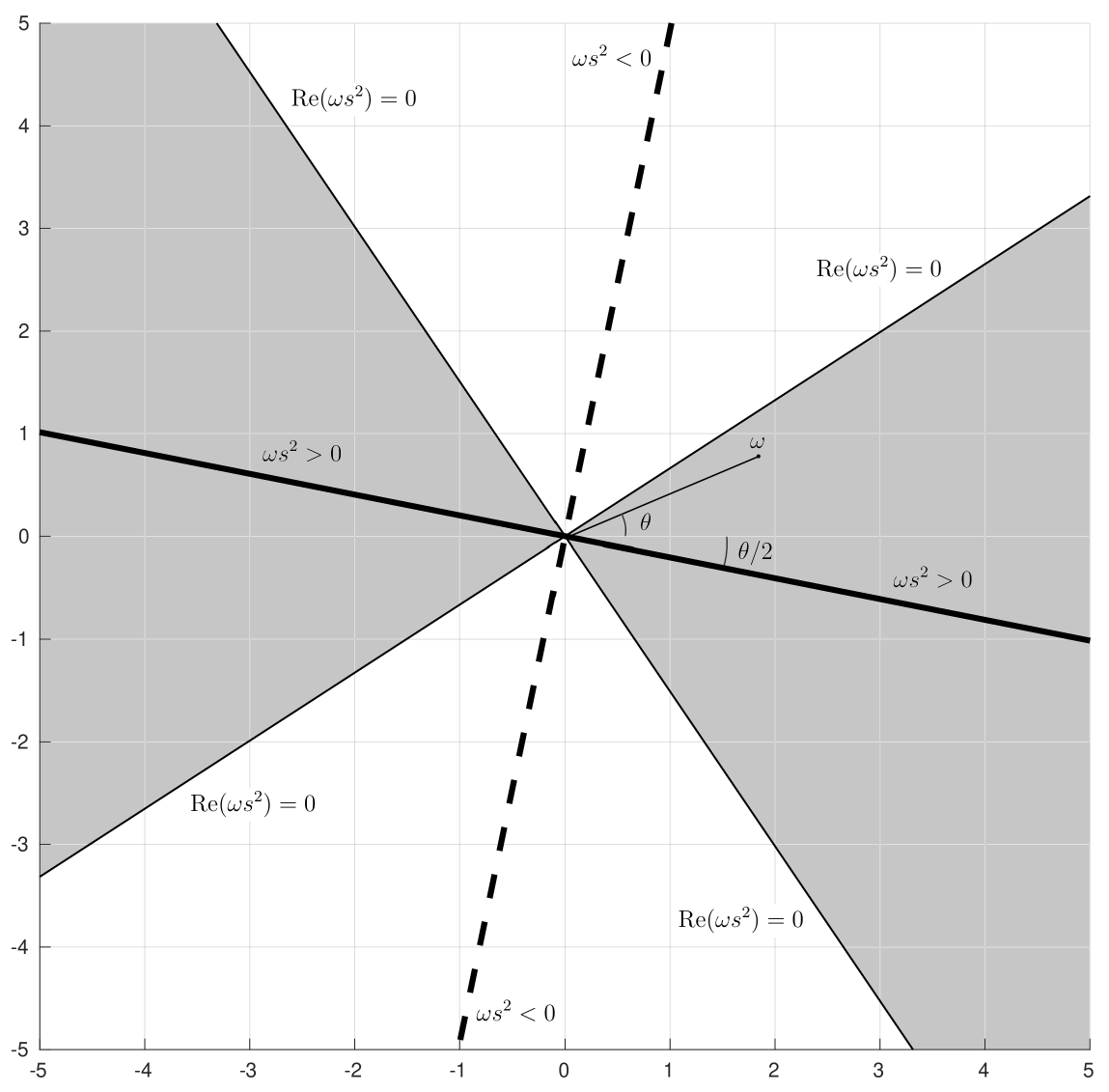}
\caption{Mapping the complex plane by the function $\omega s$ (left graph) and $ \omega s^2$ (right graph).} 
\label{fig2}
\end{figure}

Now we are ready to introduce the following seven cases: 
\label{seven_cases}
 \begin{itemize}
\item[(i)] {\it  Log-quadratic case:} If $N>0$, then $\ln(\phi(s))$ has dominant asymptotic term $N s^2 \ln(s)/(2\tau)$ in both sectors ${\mathcal S}^{\pm}_{\epsilon}$. We choose any $\theta^+ \in (\pi/4,3\pi/4)$ and 
$\theta^- \in (-3\pi/4, -\pi/4)$. It is clear that $\re(s^2)<0$ for $\arg(s)=\theta^{\pm}$.   

\item[(ii)] {\it Quadratic case:} If $N=0$, $\re(\alpha)\ge 0$ and $4\alpha \neq \pm (p-q)\i$, then  $\ln(\phi(s))$ has  quadratic dominant asymptotic terms in both sectors ${\mathcal S}^{\pm}_{\epsilon}$. We choose any $\theta^+ \in (0,\pi)$ such that $\re((\alpha+(p-q)\i/4)s^2)<0$ along the ray $\arg(s)=\theta^+$ and 
we choose any $\theta^- \in(-\pi,0)$ such that $\re((\alpha-(p-q)\i/4)s^2)<0$ along the ray $\arg(s)=\theta^-$.

\item[(iii)] {\it Upper-balanced case:} $N=0$, $4\alpha = (q-p) \i\ne0$ and 
$\arg(\mu) \neq \pi/2$. In this case $\ln(\phi(s))$ has a quadratic dominant asymptotic term in the ${\mathcal S}^{-}_{\epsilon}$ and  a dominant term of order $s \ln(s)$ in 
${\mathcal S}^{+}_{\epsilon}$. We choose any $\theta^+ \in (0,\pi)$ such that 
$\re(\mu s)>0$ along the ray $\arg(s)=\theta^+$ and we choose any $\theta^- \in(-\pi,0)$ such that $\re(\i(q-p)s^2)<0$ along the ray $\arg(s)=\theta^-$.

\item[(iv)] {\it Lower-balanced case:} $N=0$,  $4\alpha = (p-q) \i\ne0$ and 
$\arg(\mu) \neq -\pi/2$.  In this case $\ln(\phi(s))$ has a quadratic dominant asymptotic term in ${\mathcal S}^{+}_{\epsilon}$ and  a dominant term of order $s \ln(s)$ 
in ${\mathcal S}^{-}_{\epsilon}$.  We choose any $\theta^+ \in (0,\pi)$ such that 
 $\re(\i(p-q)s^2)<0$ along the ray $\arg(s)=\theta^+$ and we choose any $\theta^- \in(-\pi,0)$ such that $\re(\mu s)>0$ along the ray $\arg(s)=\theta^-$.

\item[(v)] {\it Balanced case:} $N=\alpha=p-q=0$ and $\re(\mu) \neq 0$. In this case $\ln(\phi(s))$ has dominant asymptotic terms of order $s\ln(s)$ in both sectors ${\mathcal S}^{\pm}_{\epsilon}$. We choose any $\theta^+ \in (0,\pi)$ 
and any $\theta^- \in(-\pi,0)$ such that $\re(\mu s)>0$ along the rays $\arg(s)=\theta^{\pm}$.

\item[(vi)] {\it Linear case:} $N=\alpha=p-q=\mu=0$ and $\re(\nu)<0$ --  in this case $\ln(\phi(s))$ has dominant terms of order $s$ in both sectors ${\mathcal S}^{\pm}_{\epsilon}$. We take $\theta^+=\theta^-=\pi/2$.

\item[(vii)]  {\it Logarithmic case:} $N=\alpha=p-q=\mu=\nu=0$ --  in this case $\ln(\phi(s))$ has dominant terms of order $\ln|s|$ in both sectors ${\mathcal S}^{\pm}_{\epsilon}$. We choose $\theta^{\pm}=\pm \epsilon$ if $z>1$ and $\theta^{\pm}=\pm (\pi-\epsilon)$ if $0<z<1$ (for any $\epsilon \in (0,\pi/2)$). 
\end{itemize}

Above we have described all possible choices of $\theta^{\pm}$. Some specific convenient choices of $\theta^{\pm}$ which satisfy the above conditions can be found in Table \ref{tab:def}.

\begin{definition}\label{def_K_function}
Let $m,n,p,q$ be non-negative integers,  ${\bf a}=(a_1,\dots, a_p) \in {\mathbb C}^p$,
${\bf b}=(b_1,\dots, b_q) \in {\mathbb C}^q$, $\tau>0$ and $\alpha \in \c$. Assume that $\Lambda^+ \cap \Lambda^- = \emptyset$. 
We define the Meijer-Barnes $K$-function as 
\begin{equation}\label{eq:def_K_function}
K_{p,q}^{m,n}\Big( 
\begin{matrix}
{\bf a} \\  {\bf b} 
\end{matrix} \Big
\vert z; \tau, \alpha
\Big ):=\frac{1}{2\pi \i} \int\limits_{\gamma} 
\G_{p,q}^{m,n}\Big(\begin{matrix} {\bf a} \\
{\bf b} \end{matrix}
\Big\vert s; \tau\Big)e^{\frac{\pi \alpha}{\tau}{s^2}}z^{-s}ds,
\end{equation}
where the contour of integration $\gamma={\mathcal L}_{\theta^-,\theta^+}$ is specified as in Table \ref{tab:def}. In cases (i)-(vi) the $K$-function is defined for $z>0$ and 
in case (vii) it is defined for $z\in (0,1) \cup (1,\infty)$ (note that the choice of the contour of integration in this case depends on $z$). 
 When $\alpha=0$, we will write 
$K_{p,q}^{m,n}\Big( 
\begin{matrix}
{\bf a} \\  {\bf b} 
\end{matrix} \Big
\vert z; \tau
\Big )$ instead of $K_{p,q}^{m,n}\Big( 
\begin{matrix}
{\bf a} \\  {\bf b} 
\end{matrix} \Big
\vert z; \tau, 0
\Big ).$  
\end{definition}

The $K$-function is a generalization of the Meijer $G$-function. To see this, we use the following identity
\begin{equation}\label{gamma_as_G}
\Gamma(s)=(2\pi)^{(1-\tau)/2} \tau^{s-1/2} \frac{G(s+\tau;\tau)}{G(s;\tau)},
\end{equation}
which follows from the second formula in \eqref{eq:funct_rel_G}, and rewrite each gamma factor in 
\eqref{def:Meijer_G} as a ratio of two double gamma functions. Then formula \eqref{def:Meijer_G} becomes a special case of \eqref{eq:def_K_function}.

\begin{center}
\begin{table}
\renewcommand{\arraystretch}{1.2}
\begin{tabular}[t]{ |c|c|c| }
 \hline
 (i), (vi) &  $\theta^+=\pi/2$  & $\theta^-=-\pi/2$ \\ \hline 
  (ii) &  $\theta^+=\pi/2$ if $\re(\alpha)>0$ & $\theta^-=-\pi/2$ 
 if $\re(\alpha)>0$\\
 &  $\theta^+=3\pi/4$ if $4\i \alpha+q-p>0$ & $\theta^-=-\pi/4$ 
 if $4\i\alpha+p-q>0$\\ 
 &  $\theta^+=\pi/4$ if $4\i \alpha+q-p<0$ & $\theta^-=-3\pi/4$ 
 if $4\i\alpha+p-q<0$\\ 
 \hline 
  (iii) &  $\theta^+=\pi/2$ if $\im(\mu)<0$, &  $\theta^-=-\pi/4$ if  $p>q$, \\ 
  & $\theta^+=\epsilon$ if $\arg(\mu)\in[0,\pi/2)$, & $\theta^-=-3\pi/4$ if $p<q$ \\ 
    & $\theta^+=\pi-\epsilon$ if $\arg(\mu)\in (\pi/2,\pi]$ &  \\ 
  \hline 
    (iv)  &  $\theta^+=\pi/4$ if  $p>q$, &  $\theta^-=-\pi/2$ if $\im(\mu)>0$,\\ 
  & $\theta^+=3\pi/4$ if $p<q$  &  $\theta^-=-\epsilon$  if
  $\arg(\mu) \in (-\pi/2,0]$, \\ 
 & & $\theta^-=-\pi+\epsilon$  if
  $\arg(\mu) \in [-\pi,-\pi/2) $ \\
  \hline 
  (v)  &  $\theta^+=\pi/2$ if $\im(\mu)<0$, &  $\theta^-=-\pi/2$ if $\im(\mu)>0$, \\ 
       & $\theta^+=\epsilon$  if $\arg(\mu)\in[0, \pi/2)$, & $\theta^-=-\epsilon$ if $\arg(\mu)\in(-\pi/2,0]$, \\ 
       & $\theta^+=\pi-\epsilon$ if $\arg(\mu)\in(\pi/2,\pi]$ & $\theta^-=-\pi+\epsilon$ if $\arg(\mu)\in[-\pi,-\pi/2)$\\ 
 \hline 
 (vii) &   $\theta^+=\epsilon$ if  $z>1$,   &
  $\theta^-=-\epsilon$ if  $z>1$, 
 \\ &
  $\theta^+=\pi-\epsilon$ if $0<z<1$   &
   $\theta^-=-\pi+\epsilon$ if  $0<z<1$
 \\ 
 \hline 
\end{tabular}
\caption{Possible choice of the contour of integration $\L_{\theta^-,\theta^+}$ for the seven cases of parameters described on page \pageref{seven_cases}. }
\label{tab:def} 
\end{table}
\renewcommand{\arraystretch}{1}
\end{center}

Note that the following cases of parameters are excluded from the Definition \ref{def_K_function}: 
\begin{itemize}
    \item $N<0$,
    \item $N=0$ and $\re(\alpha)<0$,
    \item $N=0$, $4\alpha=(q-p)\i$ and $\arg(\mu)=\pi/2$,
    \item $N=0$, $4\alpha=(p-q)\i$ and $\arg(\mu)=-\pi/2$,
    \item $N=\alpha=p-q=\mu=0$ and $\re(\nu)\ge0$, $\nu\ne0$.
\end{itemize}
These parameters were excluded because they violate the conditions (a) and 
(b) for choosing $\theta^{\pm}$ that we discussed on page \pageref{conditions_a_b}. Let us consider the case $N<0$. We can choose four contours of integration $\L_{-\epsilon,\epsilon}$, 
$\L_{-\epsilon,\pi-\epsilon}$, $\L_{-\pi+\epsilon,\epsilon}$ and $\L_{-\pi+\epsilon,\pi-\epsilon}$. These contours can not be transformed into each other, because of the poles of $\phi(s)$ or because of sectors of the complex plane where $|\phi(s)|\to +\infty$ as $s\to \infty$. Thus in the case $N<0$ we can have up to four different functions that could be defined via \eqref{eq:def_K_function}. While this case may be of interest, we decided not to include it in the present paper and focus only on those cases where  the Mellin-Barnes integral \eqref{eq:def_K_function} can lead to a uniquely defined function. For the same reason we exclude the case
$N=0$ and $\re(\alpha)<0$. Next, we do not allow ${\textnormal{arg}}(\mu)=\pi/2$ when $N=0$ and $4\alpha=(q-p)\i$ since in this case there is no angle $\theta^+ \in (0,\pi)$ which would make the function $\phi(s)$ decay to zero as $s\to \infty$ along the line $\arg(s)=\theta^+$. The same argument applies to the next cases where we exclude values of $\mu$ with 
 $\arg(\mu)=-\pi/2$. Finally, let us explain why we only consider  $\re(\nu)< 0$ in the linear case. Let us denote $\nu=\nu_1+\i\nu_2$. For $s=re^{i\theta_{+}}\in\c_{+}$ we have exponential decay of the integrand in 
\eqref{eq:def_K_function} if 
$$
\cos(\theta_{+})\left(\frac{\pi\nu_2}{2\tau}-\ln|z|\right)+\sin(\theta_{+})\left(\frac{\pi\nu_1}{2\tau}+\arg(z)\right)<0,
$$
and for $s=re^{i\theta_{-}}\in\c_{-}$ we have exponential decay if
$$
-\cos(\theta_{-})\left(\frac{\pi\nu_2}{2\tau}+\ln|z|\right)-\sin(\theta_{-})\left(\frac{\pi\nu_1}{2\tau}-\arg(z)\right)<0.
$$
Since this is a complicated relation between $z$ and $s$, we consider only the relatively simple case when  $\theta_+=-\theta_{-}=\pi/2$ and $\re(\nu)=\nu_1<0$. Under this condition we have exponential decay of the integrand for $z$ in the sector
$\pi\nu_1/(2\tau)<\arg(z)<-\pi\nu_1/(2\tau).
$

\begin{remark}\label{remark1}
When applying the Mellin transform to study Mellin-Barnes type integrals, it important to be able to deform the contour of integration in \eqref{eq:def_K_function} to a straight vertical line $c+\i \r$, which is a special case of $\L_{\i \infty}$. Below we list the values of parameters and the variable $z$ allowing for the contour $\gamma$ to be chosen as $\L_{\i \infty}$:
\begin{itemize}
    \item $N>0$;
    \item $N=0$, $\re(\alpha)>0$;
    \item $N=\alpha=p-q=0$, $\mu \in \r \setminus \{0\}$, $\re(\nu)<0$, $|\arg(z)|<-\pi\re(\nu)/(2\tau)$;
    \item $N=\alpha=p-q=\mu=\nu=0$ and $\re(\xi)<-2\tau$, $z\in(0,1)\cup(1,\infty)$. 
\end{itemize}
The convergence of the integral in \eqref{eq:def_K_function} over $\L_{\i \infty}$ for the above choices can be verified directly from Proposition~\ref{prop:phi-asymp}.
\end{remark}

From now on we will call the values of parameters
${\bf a}$, ${\bf b}$, $\alpha$ and $\tau$ for which the
$K$-function is defined {\it admissible parameters}. 
In the next theorem we summarize analyticity properties of the $K$-function.

\begin{theorem}[Analyticity properties]
\label{prop_analyticity}${}$
For fixed values of admissible parameters ${\bf a}$, ${\bf b}$, $\alpha$ and $\tau$, the function 
$
w \in \r \mapsto K_{p,q}^{m,n}\Big( 
\begin{matrix}
{\bf a} \\  {\bf b} 
\end{matrix} \Big
\vert e^{2\pi w}; \tau, \alpha
\Big )
$
can be  continued analytically to an entire function in cases (i)-(v). In case (vi) the function $w \in \r \mapsto K_{p,p}^{m,n}\Big( 
\begin{matrix}
{\bf a} \\  {\bf b} 
\end{matrix} \Big
\vert e^{2\pi w}; \tau
\Big )$
can be continued analytically in the horizontal strip $|\im(w)|<-\re(\nu)/(4\tau)$. 
 In case (vii) this function can be continued analytically in each half-plane $\re(w)<0$ and $\re(w)>0$ and it is continuous on $\r$ if $\re(\xi)<-2\tau$.
\end{theorem}
\begin{proof}
Let us consider case (vi) first. The contour of integration is $\L_{\i \infty}$ and formula \eqref{eq:phi_asymptotics_case3} implies that as $|s|\to \infty$
$$
\G_{p,q}^{m,n}\Big(\begin{matrix} {\bf a} \\
{\bf b} \end{matrix}
\Big\vert s; \tau\Big)=O\Big(e^{\frac{\pi }{2\tau}\re(\nu) |s|}\Big), \;\;\; s\in \L_{\i\infty}. 
$$
From here it follows that the $K$-function defined via \eqref{eq:def_K_function} is analytic in the sector $|\arg(z)|<-\pi \re(\nu)/(2\tau)$, and after changing variables $z=\exp(2\pi w)$ we obtain the desired analyticity in the strip $|\im(w)|<-\re(\nu)/(4\tau)$. 

The cases (i)-(v) are all established in the same way: we use Proposition~\ref{prop:phi-asymp} and the prescription of the contour of integration $\L_{\theta^-,\theta^+}$ in Definition \ref{def_K_function} to check that the function $\phi(s)$ defined in \eqref{def_phi} decays faster than any exponential function as $s\to \infty$ on the contour ${\mathcal L}_{\theta^-,\theta^+}$. For example, in case (i) we have the contour $\gamma={\mathcal L}_{\i \infty}$ and we know from Proposition~\ref{prop:phi-asymp} that the dominant term in the asymptotics of $\phi(s)$ is $\exp(N/(2\tau) s^2 \ln(s))$, and when $s=\i t$ this decays as 
$$
\exp\Big( \frac{N}{2\tau} s^2 \ln(s)\Big)=
\exp\Big(- \frac{N}{2\tau} t^2 \ln |t| + O(t^2)\Big), \;\;\; t\to \infty. 
$$
This proves that the integral on the right-hand side of \eqref{eq:def_K_function} is an analytic function of $w=\ln(z)/(2\pi)$. 
Similar argument can be applied in the remaining cases  (ii)-(v), we omit the details.

Analyticity of $K_{p,p}^{m,n}\Big( 
\begin{matrix}
{\bf a} \\  {\bf b} 
\end{matrix} \Big
\vert e^{2\pi w}; \tau
\Big )$ in case (vii) follows from our choice of the contour of integration. The statement about continuity follows from the fact that if $\re(\xi)<-2\tau$ and $z>0$ then the integrand in 
\eqref{eq:def_K_function} is integrable on the contour $\L_{\i \infty}$ (due to \eqref{eq:phi_asymptotics_case4}) and we can deform each contours $\L_{-\epsilon,\epsilon}$ (when $z>1$) and  $\L_{-\pi+\epsilon,\pi-\epsilon}$ (when $0<z<1$) to the contour $\L_{\i \infty}$. Now the $K$-function is
expressed as an integral 
\eqref{eq:def_K_function} over contour $\L_{\i \infty}$, and since $\phi(s)$ is integrable on this contour, the $K$-function is a continuous function of $z>0$.
\end{proof}

In what follows, whenever we mention the domain of $K$-function, we will mean the Riemann surface of the logarithm $-\infty< \arg(z)<\infty$ in cases (i)-(v), the sector $|\arg(z)| \le - \pi \re(\nu) /(2 \tau) $ in case (vi) 
and the set $(0,1)\cup (1,\infty)$ in case (vii).

We would like to reiterate that in Table \ref{tab:def} we show only one possible way to choose the contour of integration $\L_{\theta^-,\theta^+}$, but there are many other ways to choose the contour of integration in \eqref{eq:def_K_function} (all of which are listed on page \pageref{seven_cases}). However, these different choices of $\theta^{\pm}$ would lead to the same function $K_{p,q}^{m,n}(z)$ defined via \eqref{def_K_function}. The fact that the integral on the right-hand side of \eqref{eq:def_K_function} would not depend on the choice of $\theta^{\pm}$ can be established by the classical technique of rotating the contour of integration.  Another important point is that when choosing the contour $\gamma$ in Definition \ref{def_K_function} we are free to shift the rays $s=c+re^{i\theta^{\pm}}$, $0\le r<\infty$, forming the initial and the terminal parts of the contour ${\mathcal L}_{\theta^-,\theta^+}$ to the left or to the right (i.e. replace $c$ by $c'$ with $\im(c)=\im(c')$). What matters for the convergence of the integral in \eqref{eq:def_K_function} is only the slopes $\theta^{\pm}$ of these rays. This fact will be useful in the proof of the next result, where we will need to shift the contour of integration.

\begin{theorem}[Transformation properties]\label{thm_transformation}
For admissible parameters ${\bf a}$, ${\bf b}$, $\alpha$, $\tau$ and $z$ in the domain of analyticity, we have the following three transformation properties: 
\begin{equation}\label{eq:transformation_K_shift_by_c}
 K_{p,q}^{m,n}\Big( 
\begin{matrix}
c+{\bf a} \\  c+{\bf b} 
\end{matrix} \Big
\vert z; \tau, \alpha
\Big )=
e^{\frac{\pi \alpha}{\tau} c^2}  z^{-c}  K_{p,q}^{m,n}\Big( 
\begin{matrix}
{\bf a} \\  {\bf b} 
\end{matrix} \Big
\vert e^{-\frac{2\pi \alpha}{\tau} c} z; \tau, \alpha
\Big ),  \;\;\; c\in \c,
 \end{equation}
\begin{equation}\label{eq:tranformation_K_z_1/z}
 K_{p,q}^{m,n}\Big( 
\begin{matrix}
{\bf a} \\  {\bf b} 
\end{matrix} \Big
\vert z; \tau, \alpha
\Big )=
 K_{q,p}^{n,m}\Big( 
\begin{matrix}
1+\tau-{\bf b} \\ 1+\tau- {\bf a} 
\end{matrix} \Big \vert
 z^{-1}; \tau, \alpha
\Big ),
 \end{equation}
 \begin{equation}\label{eq:transformation_K_z_power_tau}
 K_{p,q}^{m,n}\Big( 
\begin{matrix}
{\bf a} \\  {\bf b} 
\end{matrix} \Big
\vert z; \tau, \alpha
\Big )=
A \tau  K_{p,q}^{m,n}\Big( 
\begin{matrix}
\tau^{-1}{\bf a} \\  \tau^{-1}{\bf b} 
\end{matrix} \Big
\vert  (Bz)^{\tau}; \tau^{-1}, \alpha-\frac{N}{2\pi} \ln(\tau)
\Big ),
 \end{equation}
 where $A$ and $B$ as defined in \eqref{def_A_B} are given by
\begin{equation*}
A=(2\pi)^{\frac{1-\tau}{2\tau}(\nu-(1+\tau)(n+m-q))}
\tau^{\frac{1}{2\tau} ((1+\tau)\mu-\xi)-N}, \;\;\; 
B=(2\pi)^{\frac{1-\tau}{2\tau}(p-q)} \tau^{\frac{1}{2\tau}(N(1+\tau)-2\mu)}.
\end{equation*}
\end{theorem}

Note that the Meijer $G$-function also enjoys properties similar to \eqref{eq:transformation_K_shift_by_c}
 and \eqref{eq:tranformation_K_z_1/z}, see formulas 9.31.2 and 9.31.5 in \cite{Jeffrey2007}. Formula 
 \eqref{eq:transformation_K_z_power_tau} is unique to $K$-function and has no analogue in the case of Meijer $G$-function. This last transformation formula is a consequence of the modular transformation property \eqref{eq:G_1_over_tau} of the double gamma function. 

\vspace{0.25cm}
\noindent
{\bf Proof of Theorem \ref{thm_transformation}:}
We prove \eqref{eq:transformation_K_shift_by_c} first. 
Let $\gamma={\mathcal L}_{\theta^+,\theta^-}$ be a contour of integration from Definition \ref{def_K_function} corresponding to parameters ${\mathbf a}$, ${\mathbf b}$. Then $\gamma+c$ is a contour corresponding to parameters $\tilde {\mathbf a}=c+{\mathbf a}$, 
$\tilde {\mathbf b}=c+{\mathbf b}$.
Also, if $\mu$ and $\nu$ are given by \eqref{eq:parameters} for parameters ${\mathbf a}$, ${\mathbf b}$, then the corresponding values for  parameters $\tilde {\mathbf a}$, 
$\tilde {\mathbf b}$ are 
$$
\tilde \mu = Nc+\mu, \;\;\; \tilde \nu = (p-q)c+\nu, 
$$
which shows that the seven cases discussed on page \pageref{seven_cases} are preserved under this change of parameters. 
Therefore we can write
 \begin{align*}
K_{p,q}^{m,n}\Big( 
\begin{matrix}
c+{\bf a} \\  c+{\bf b} 
\end{matrix} \Big
\vert z; \tau, \alpha
\Big )&=\frac{1}{2\pi \i} \int\limits_{\gamma+c} 
\G_{p,q}^{m,n}\Big(\begin{matrix} c+{\bf a} \\
c+{\bf b} \end{matrix}
\Big\vert s; \tau\Big)e^{\frac{\pi \alpha}{\tau}{s^2}}z^{-s}\d s\\
&=\frac{1}{2\pi \i} \int\limits_{\gamma+c} 
\G_{p,q}^{m,n}\Big(\begin{matrix} {\bf a} \\
{\bf b} \end{matrix}
\Big\vert s-c; \tau\Big)e^{\frac{\pi \alpha}{\tau}{s^2}}z^{-s}\d s\\
&=\frac{1}{2\pi \i} \int\limits_{\gamma} 
\G_{p,q}^{m,n}\Big(\begin{matrix} {\bf a} \\
{\bf b} \end{matrix}
\Big\vert v; \tau\Big)e^{\frac{\pi \alpha}{\tau}{(v+c)^2}}z^{-(v+c)}\d v
\\&
=e^{\frac{\pi \alpha}{\tau} c^2}  z^{-c}  K_{p,q}^{m,n}\Big( 
\begin{matrix}
{\bf a} \\  {\bf b} 
\end{matrix} \Big
\vert e^{-\frac{2\pi \alpha}{\tau} c} z; \tau, \alpha
\Big ).
\end{align*}
In the above computation, we used the second identity in \eqref{eq:G_identity1} in step 2 and we changed the variable of integration  $s=v+c$ in step 3. 
Formula \eqref{eq:tranformation_K_z_1/z} is established in the same way, by using the first identity in \eqref{eq:G_identity1} and changing the variable of integration $s\mapsto -s$. The third transformation \eqref{eq:transformation_K_z_power_tau} can be proved using \eqref{eq:G_s_over_tau_transformation} and changing the variable of integration $s\mapsto v \tau$ in \eqref{eq:def_K_function}, the details are omitted.    
\qed

\begin{theorem}\label{thm_logarithmic_case}
Assume that $p\ge 1$ and the parameters ${\mathbf a}=(a_1,\dots,a_p)\in \c^p$ and ${\mathbf b}=(b_1,\dots,b_p) \in \c^p$ satisfy $\sum_{j=1}^p a_j=\sum_{j=1}^p b_j$. Then 
\begin{equation}\label{K_is_zero_1}
K_{p,p}^{0,p}\Big( 
\begin{matrix}
{\bf a} \\  {\bf b} 
\end{matrix} \Big
\vert z; \tau
\Big )=0, \;\;\; {\textnormal{ for }} \; z>1,  
\end{equation}
and 
\begin{equation}\label{K_is_zero_2}
K_{p,p}^{p,0}\Big( 
\begin{matrix}
{\bf a} \\  {\bf b} 
\end{matrix} \Big
\vert z; \tau
\Big )=0, \;\;\; {\textnormal{ for }} \; 0<z<1.  
\end{equation}
\end{theorem}
\begin{proof}
Let us prove \eqref{K_is_zero_1}. We check that we are in the logarithmic case, since $N=p-q=\mu=\nu=0$. 
 The conditions imply that the integrand in \eqref{eq:def_K_function}
has no poles to the right of the contour of integration $\gamma$ (in other words, the set $\Lambda^+$ in \eqref{def:Lambda_pm} is empty). 
Assuming that $z>1$ and
using the Phragmen-Lindel\"of Theorem  (see Corollary 4.2 on page 139 in \cite{Conway1978}) and the fact that the double gamma function is entire of order 2, the integrand in \eqref{eq:def_K_function} decays exponentially as $s\to \infty$, uniformly inside the contour $\L_{-\epsilon,\epsilon}$. Thus we can shift this contour of integration $\L_{-\epsilon,\epsilon} \mapsto c+\L_{-\epsilon,\epsilon}$ for any $c>0$, and taking $c \to +\infty$ the integral in \eqref{eq:def_K_function} will converge to zero (due to the exponential decay of the integrand). This ends the proof of \eqref{K_is_zero_1}. The second result \eqref{K_is_zero_2} can be obtained from 
\eqref{K_is_zero_1} by applying transformation formula \eqref{eq:tranformation_K_z_1/z}. 
\end{proof}

\section{Algebraic asymptotics of $K$ function}\label{section_K_asymptotics}

Suppose $f(s)$ is meromorphic in $\mathbb{C}$ and $F(z)$ is defined by the Mellin-Barnes integral
\begin{equation}\label{eq:F-defined}
	F(z)=\int_{\gamma}f(s)z^{-s}\d s,
\end{equation}
which is absolutely convergent for all $z$ in some sector 
\begin{equation}\label{eq:z-sector}
	S_{\alpha,\beta}=\{z\!:-\infty\le\alpha<\arg(z)<\beta\le\infty\}~(\text{if}~\alpha<\beta),\;\; \text{or} \;\; ~S_{\alpha,\alpha}=\{z\!: z=re^{i\alpha},~r>0\}
\end{equation}
on the Riemann surface of the logarithm. Assume that the integration contour $\gamma$ is a simple loop passing through the point at infinity (stereographic projection of a closed Jordan curve passing through the north pole) which lies entirely in a left half-plane $\re(s)<\sep$ for some real $\sep$, and that $\gamma$ avoids the poles of $f$. 
By the Jordan curve theorem, this contour divides the plane into two simply connected domains.  We will call the domain containing the half-plane $\re(s)>\sep$ the exterior domain and the other one -- the interior domain with respect to $\gamma$.  This induces a positive direction on $\gamma$ -- the one that leaves the  interior domain on the left. We will assume this to be the direction of integration in \eqref{eq:F-defined}.

Denote by $\Lambda_{-}$ ($\Lambda_{+}$) the set of poles of $f$ in the interior (exterior) domain.  Denote by $\Lambda_{-}^{\sharp}$ the subset of elements of  $\Lambda_{-}$ having the maximal real part, for which we will use the symbol $\lambda_{\sharp}$, and assume that $\Lambda_{-}^{\sharp}$ is finite, say $\#\Lambda_{-}^{\sharp}=r$.  We will make the following assumptions regarding the asymptotic behavior of $f$ (in addition to the absolute convergence of the integral in \eqref{eq:F-defined}):

\smallskip

(I) either there exists a semi-strip $\{s: \im(s)>t_1, \sigma_{\sharp}<\re(s)<\sep~\text{with}~\sigma_{\sharp}<\lambda_{\sharp}\}$ free from the points of $\gamma$  or there is a sequence  $\eta_k\to+\infty$ such that
$$
\lim\limits_{k\to\infty}|f(\sigma+i\eta_k)e^{\eta_k\beta}|=0
$$ 
uniformly in $\sigma\in[\sigma_{\sharp},\sep]$, where  $\beta$ is defined in \eqref{eq:z-sector}; 

\smallskip

(II) either there exists a semi-strip $\{s: \im(s)<t_2, \sigma_{\sharp}<\re(s)<\sep~\text{with}~\sigma_{\sharp}<\lambda_{\sharp}\}$ free from the points of $\gamma$  or there is a sequence  $\eta_k\to+\infty$ such that
$$
\lim\limits_{k\to\infty}|f(\sigma-i\eta_k)e^{-\eta_k\alpha}|=0
$$ 
uniformly in $\sigma\in[\sigma_{\sharp},\sep]$, where  $\alpha$ is defined in \eqref{eq:z-sector}.

\smallskip

\begin{lemma}\label{lm:F_zero_asymp}
	Denote the elements of $\Lambda_{-}^{\sharp}$ by $p_1,p_2,\ldots,p_r$, $\re(p_j)=\lambda_{\sharp}$ for $j=1,\ldots,r$, and write $m_1,m_2,\ldots,m_r$ for their corresponding multiplicities. Under the above assumptions, the following asymptotic approximation holds as $z\to0$ in the sector $S_{\alpha,\beta}$, $\alpha\le\beta$,
	\begin{equation}\label{eq:F_zero_asymp}
		F(z)\sim \bigg[\sum\limits_{k=1}^{r} z^{-p_k}\sum\limits_{i=0}^{m_k-1}A_{k,i}\ln^i(z)\bigg](1+O(|z|^{-\lambda_{\sharp}+\delta}))
	\end{equation}
	for some $\delta>0$ and some complex constants $A_{k,i}$. The approximation is uniform in any closed subsector of the sector $S_{\alpha,\beta}$.
\end{lemma}
\begin{proof}  Choose any $\sigma_{*}\in(\sigma_{\sharp},\lambda_{\sharp})$  satisfying 
$$
\max\limits_{p\in\Lambda_{-}\setminus\Lambda_{-}^{\sharp}}(\re(p))<\sigma_{*}
$$
(which exists in view of $\sigma_{\sharp}<\lambda_{\sharp}$).  We now delete the part of the contour $\gamma$ with $\re(s)>\sigma_{*}$ and connect the remaining parts by straight line segments lying on the line  $\re(s)=\sigma_{*}$. Denote the new contour by $\gamma'$. By construction the only singularities of $f$ lying between $\gamma$ and $\gamma'$ are the poles from the set $\Lambda_{-}^{\sharp}$. Hence, by Cauchy's theorem
$$
F(z)=\int_{\gamma}f(s)z^{-s}\d s=\sum\limits_{k=1}^{r}\res\limits_{s=p_k}\big[f(s)z^{-s}\big]+\int_{\gamma'}f(s)z^{-s}\d s
$$
if there are no points of $\gamma$ in both semi-strips defined in (I) and (II) above. Indeed, in this case the contour is only altered in the rectangle $\{\sigma_{*}<\re(s)<\sep, t_2<\im(s)<t_1\}$ not affecting the convergence of the integral. In the case when the  upper and/or lower semi-strip contains the points of $\gamma$,  the replacement of $\gamma$ by $\gamma'$ is justified by the estimates using the limit from (I): 
\begin{multline*}
	\left|\int_{\sigma_{\sharp}}^{\sep}f(\sigma+i\eta_k)z^{-\sigma-i\eta_k} \d \sigma\right|\le
	e^{\eta_k\arg(z)}\int_{\sigma_{\sharp}}^{\sep}|f(\sigma+i\eta_k)||z|^{-\sigma}\d \sigma
	\\
	\le (\sep-\sigma_{\sharp})e^{\eta_k\beta}\max\limits_{\sigma_{\sharp}\le\sigma\le\sep}|f(\sigma+i\eta_k)||z|^{-\sigma}\to0~\text{as}~k\to\infty.
\end{multline*}
 and similar estimates in the lower half plane using the limit in (II). Now for all $s\in \gamma'$ we have $\re(s)\le\sigma_{*}$. This implies that for all sufficiently small $|z|$ and all $s\in \gamma'$:
$$
|z^{-s}|\le|z|^{-\sigma_{*}}e^{\arg(z)\im(s)}.
$$
Then we have, in view of the absolute convergence,
$$
\bigg|\int_{C'}f(s)z^{-s}\d s\bigg|\le
\int_{C'}|f(s)||z^{-s}||\d s|
\le|z|^{-\sigma_{*}}\int_{C'}|f(s)|e^{\im(s)\arg(z)}||\d s|\le M|z|^{-\sigma_{*}}.
$$
Computing the residues we finally arrive at \eqref{eq:F_zero_asymp}.
\end{proof}

\begin{remark} The standard residue formula gives the following expression for the constants $A_{k,i}$ 
$$
A_{k,i}=\binom{m_k-1}{i}(-1)^i\lim\limits_{s\to p_k}[(s-p_k)^{m_k}f(s)]^{(m_k-i)}.
$$
\end{remark}

Assume now that the integration contour $\gamma$ in  \eqref{eq:F-defined} is a simple loop passing through the point at infinity (stereographic projection of a a closed Jordan curve passing through the north pole) which lies entirely in the right half-plane $\re(s)>\sep$ for some real $\sep$ and avoids the poles of $f$. Then the half-plane $\re(s)<\sep$ belongs to the exterior domain and the poles from $\Lambda_{+}$  belong to the interior domain. 
We now choose negative direction of integration with respect to the interior domain (this is motivated by the choice of contour in Definition~\ref{def_K_function}).
Writing $\Lambda_{+}^{\flat}$ for the subset of poles from  $\Lambda_{+}$ having the minimal real part denoted $\lambda_{\flat}$, we assume again that $\Lambda_{+}^{\flat}$ is finite, say $\#\Lambda_{+}^{\flat}=r$. The conditions (I) and (II) should be replaced by 

(I') either there exists a semi-strip $\{s: \im(s)>t_1, \sep<\re(s)<\sigma_{\flat}~\text{with}~\lambda_{\flat}<\sigma_{\flat}\}$ free from the points of $\gamma$  or there is a sequence  $\eta_k\to+\infty$ such that
$$
\lim\limits_{k\to\infty}|f(\sigma+i\eta_k)e^{\eta_k\beta}|=0
$$ 
uniformly in $\sigma\in[\sep,\sigma_{\flat}]$;

\smallskip

(II') either there exists a semi-strip $\{s: \im(s)<t_2, \sep<\re(s)<\sigma_{\flat}~\text{with}~\lambda_{\flat}<\sigma_{\flat}\}$ free from the points of $\gamma$  or there is a sequence  $\eta_k\to+\infty$ such that
$$
\lim\limits_{k\to\infty}|f(\sigma-i\eta_k)e^{-\eta_k\alpha}|=0
$$ 
uniformly in $\sigma\in[\sep,\sigma_{\flat}]$.

Writing $w=1/z$ and making the substitution $s\to -s$ in \eqref{eq:F-defined} we obtain 
$$
\tilde{F}(w):=F(1/w)=\int_{\tilde{\gamma}}f(-s)w^{-s}\d s
$$
where $\tilde{\gamma}=-\gamma$ is the reflection of $\gamma$ with respect to the origin.  Now the roles of the sets $\Lambda_{+}$ and $\Lambda_{-}$ are interchanged and the function $\tilde{F}(w)$ satisfies the conditions of Lemma~\ref{lm:F_zero_asymp} with opposite direction of integration, yielding
\begin{lemma}\label{lm:F_inf_asymp} 
	Denote the elements of $\Lambda_{+}^{\flat}$ by $q_1,q_2,\ldots,q_r$, $\re(q_j)=\lambda_{\flat}$, $j=1,\ldots,r$, and write $n_1,n_2,\ldots,n_r$ for their corresponding multiplicities. Under the above assumptions, the following asymptotic approximation holds as $z\to\infty$ in the sector $S_{\alpha,\beta}$, $\alpha\le\beta$,
	$$
	F(z)\sim -\bigg[\sum\limits_{k=1}^{r} z^{-q_k}\sum\limits_{i=0}^{n_k-1}B_{k,i}\ln^i(z)\bigg](1+O(|z|^{-\lambda_{\flat}-\delta}))
	$$
	for some $\delta>0$ and some complex constants $B_{k,i}$. The approximation is uniform in any closed subsector of the sector $S_{\alpha,\beta}$.
\end{lemma}

\bigskip

Next, we will apply Lemmas~\ref{lm:F_zero_asymp} and \ref{lm:F_inf_asymp} to $K$ function.  This leads to the following statements. 

\begin{theorem}\label{thm4}
	Let
	$$
	b_{\sharp}=\max_{m+1\le j\le q}(\re(b_j))
	$$
	and write $\Lambda_{-}^{\sharp}$ for the subset of elements of $\Lambda_{-}$ having the real part $\lambda_{\sharp}=b_{\sharp}-\tau-1$. Denote these elements by $p_1,p_2,\ldots,p_r$ and write $m_1,m_2,\ldots,m_r$ for their corresponding multiplicities. 
Suppose Table~\ref{tab:def} allows for the contour $\L_{\theta^{-},\theta^{+}}$ with $\theta^{+}\ge\pi/2$ and $\theta^{-}\le-\pi/2$  in \eqref{eq:def_K_function}. 
	Then there is $\delta>0$ such that as $z\to0$
	$$
	K_{p,q}^{m,n}\Big( 
	\begin{matrix}
		{\bf a} \\  {\bf b} 
	\end{matrix} \Big
	\vert z; \tau, \alpha
	\Big )= \bigg[\sum\limits_{k=1}^{r} z^{-p_k}\sum\limits_{i=0}^{n_k-1}A_{k,i}\ln^i(z)\bigg](1+O(|z|^{-\lambda_{\sharp}+\delta}))
	$$
	for some complex constants  $A_{k,i}$,	where $|\arg(z)|<-\pi\re(\nu)/(2\tau)$ in case (vi) and $\arg(z)=0$ in case (vii).	
\end{theorem}
\textbf{Remark.} If the pole at $s=p_k$ is simple, i.e. $m_k=1$, then we have
$$
A_{k,0}=\frac{\tau\prod\limits_{j=1}^m G(b_j-p_k;\tau)\prod\limits_{j=1}^n G(1+\tau-a_j+p_k;\tau)}
{ \sideset{}{'}\prod\limits_{j=m+1}^q G(1+\tau-b_j+p_k;\tau)\prod\limits_{j=n+1}^p G(a_j-p_k;\tau)}  
e^{\frac{\pi\alpha}{\tau}{p_k^2}}.
$$

In a similar fashion we get the following result (which can also be obtained from from Theorem \ref{thm4} by applying \eqref{eq:tranformation_K_z_1/z}).

\begin{theorem}\label{thm5}
	Let
	$$
	a_{\flat}=\min_{n+1\le j\le p}(\re(a_j))
	$$
	and write $\Lambda_{+}^{\flat}$ for the subset of elements of $\Lambda_{+}$ having the real part $a_{\flat}$. Denote these elements by $q_1,q_2,\ldots,q_r$ and write $n_1,n_2,\ldots,n_r$ for their corresponding multiplicities. 
	Suppose Table~\ref{tab:def} allows for the contour $\L_{\theta^{-},\theta^{+}}$ with $\theta^{+}\le\pi/2$ and $\theta^{-}\ge-\pi/2$. 
	Then there is $\delta>0$ such that as $z\to\infty$
	$$
	K_{p,q}^{m,n}\Big( 
	\begin{matrix}
		{\bf a} \\  {\bf b} 
	\end{matrix} \Big
	\vert z; \tau, \alpha
	\Big )= -\bigg[\sum\limits_{k=1}^{r} z^{-q_k}\sum\limits_{i=0}^{m_k-1}B_{k,i}\ln^i(z)\bigg](1+O(|z|^{-a_{\flat}-\delta}))
	$$
	for some complex constants  $B_{k,i}$,	where $|\arg(z)|<-\pi\re(\nu)/(2\tau)$ in case (vi) and $\arg(z)=0$ in case (vii).		
\end{theorem}

\noindent
\textbf{Remark.} If the pole at $s=q_k$ is simple, i.e. $n_k=1$, then we have
$$
B_{k,0}=\frac{\tau\prod\limits_{j=1}^m G(b_j-q_k;\tau)\prod\limits_{j=1}^n G(1+\tau-a_j+q_k;\tau)}
{ \prod\limits_{j=m+1}^q G(1+\tau-b_j+q_k;\tau)\sideset{}{'}\prod\limits_{j=n+1}^p G(a_j-q_k;\tau)}  
e^{\frac{\pi\alpha}{\tau}{q_k^2}}.
$$

\section{The Mellin transform and integral equations satisfied by $K$ function}\label{section_Mellin_K_function}

We open this section with the conditions for existence of the Mellin transform of $K$ function. The standard Mellin inversion theorem \cite[Theorem~28]{Titchmarsh} immediately yields. 
\begin{theorem}
	Suppose	$\max\limits_{s\in\Lambda_{-}}\re(s)<\min\limits_{s\in\Lambda_{+}}\re(s)$ which amounts to 
	\begin{equation}\label{eq:KMellinStrip}
	\lambda_{\sharp}:=\max_{m+1\le j\le q}(\re(b_j))-\tau-1<\lambda_{\flat}:=\min_{n+1\le j\le p}(\re(a_j))
	\end{equation}
	and assume that conditions listed in Remark~\ref{remark1} hold true, so that $\gamma=\L_{\i\infty}$ in \eqref{eq:def_K_function}. Then the Mellin transform of $K$ function exists and 
	$$
	\int_0^{\infty}K_{p,q}^{m,n}\Big( 
	\begin{matrix}
		{\bf a} \\  {\bf b} 
	\end{matrix} \Big
	\vert x; \tau, \alpha
	\Big )x^{s-1}\d x=e^{\frac{\pi \alpha}{\tau}{s^2}}\G_{p,q}^{m,n}\Big(\begin{matrix} {\bf a} \\
		{\bf b} \end{matrix}
	\Big\vert s; \tau\Big)
	$$
	for $s$ in the strip $\lambda_{\sharp}<\re(s)<\lambda_{\flat}$.
\end{theorem}

Define two auxiliary functions in terms Meijer's $G$ function 
\begin{equation}\label{eq:psi-defined}
f(t):=\tau t^{\tau+1}G^{n,p-n}_{p+q-m-n,m+n}\Bigg( 
t^{\tau}\Big \vert
\begin{matrix}
	-a_{n+1}/\tau,\ldots,-a_{p}/\tau,-b_{m+1}/\tau,\ldots,-b_{q}/\tau\\
	-a_1/\tau,\ldots,-a_{n}/\tau,-b_1/\tau,\ldots,-b_{m}/\tau
\end{matrix} \Bigg),
\end{equation}
\begin{equation}\label{eq:phi-defined}
h(t):=t^{\tau+1}G^{n,p-n}_{p+q-m-n,m+n}\Big( 
t\Big \vert
\begin{matrix}
	-a_{n+1},\ldots,-a_{p},-b_{m+1},\ldots,-b_{q}\\
	-a_1,\ldots,-a_{n},-b_1,\ldots,-b_{m}
\end{matrix}  \Big).
\end{equation}
Using quasi-periodicity of Barnes $G$ function from \eqref{eq:funct_rel_G}  we arrive at
\begin{theorem}\label{th:firstinteq}
Suppose $p\ge{q}$ and the vertical strips defined by conditions \eqref{eq:KMellinStrip} and 
\begin{equation}\label{eq:GMellinStrip1}
\max\limits_{1\le{i}\le{n}}\re(a_i)-\tau-1<\re(s)<\min\limits_{n+1\le{i}\le{p}}\re(a_i)-1
\end{equation}
\emph{(}with additional restriction  $\re(s)<(\re(\mu)-\tau)/N-\tau/2-1$ if $p=q$\emph{)}
have nonempty intersection.
Then $K$ function satisfies the  integral equation
\begin{equation}\label{eq:firstinteq}
e^{-\frac{\pi\alpha}{\tau}}xK_{p,q}^{m,n}\Big(\begin{matrix}
		{\bf a} \\  {\bf b}\end{matrix} \Big\vert x; \tau, \alpha\Big)=
	\int_{0}^{\infty}f\big(e^{-\frac{2\pi\alpha}{\tau}}x/t\big)
	K_{p,q}^{m,n}\Big(\begin{matrix}
		{\bf a} \\  {\bf b}\end{matrix} \Big\vert t; \tau, \alpha\Big)\frac{\d t}{t}.    
\end{equation}
In a similar fashion, if the vertical strips defined by conditions \eqref{eq:KMellinStrip} and 
\begin{equation}\label{eq:GMellinStrip2}
\max\limits_{1\le{i}\le{n}}\re(a_i)-\tau-1<\re(s)<\min\limits_{n+1\le{i}\le{p}}\re(a_i)-\tau
\end{equation}
\emph{(}with additional restriction $\re(s)<(\re(\mu)-1)/N-\tau-1/2$ if $p=q$\emph{)} have nonempty intersection, then 
\begin{equation}\label{eq:secondinteq}
	(2\pi)^{(\tau-1)(q-p)/2}\tau^{N(\tau+1/2)-\mu}e^{-\pi\alpha\tau}x^{\tau}K_{p,q}^{m,n}\Big(\begin{matrix}
		{\bf a} \\  {\bf b}\end{matrix} \Big\vert x; \tau, \alpha\Big)=
	\int_{0}^{\infty}h\big(\tau^{N}e^{-2\pi\alpha}x/t\big)
	K_{p,q}^{m,n}\Big(\begin{matrix}
		{\bf a} \\  {\bf b}\end{matrix} \Big\vert t; \tau, \alpha\Big)\frac{\d t}{t}.
\end{equation}
\end{theorem}
\begin{proof}
The first quasi-periodicity of Barnes $G$ function in  \eqref{eq:funct_rel_G} leads to the following relation
$$
\G_{p,q}^{m,n}(s+1)=\underbrace{\frac{\prod_{j=1}^{n}\Gamma\Big(1+1/\tau-a_j/\tau+s/\tau\Big)\prod_{j=n+1}^{p}\Gamma\Big(a_j/\tau-s/\tau-1/\tau\Big)}{\prod_{j=1}^{m}\Gamma\Big(b_j/\tau-s/\tau-1/\tau\Big)\prod_{j=m+1}^{q}\Gamma\Big(1+1/\tau-b_j/\tau+s/\tau\Big)}}_{=F(s)}\G_{p,q}^{m,n}(s).
$$
Basic properties of the Mellin transform and  Meijer's $G$ function show that
$\mathcal{M}[f](s)=F(s)$ with $f$ defined in \eqref{eq:psi-defined}.  Suppose now that $\lambda_{\flat}-\lambda_{\sharp}>1$ and the integral over $\gamma=\L_{\i\infty}$ converges. Then, we can write by changing the integration variable $s\to{s+1}$ in \eqref{eq:def_K_function}
\begin{multline*}
	K_{p,q}^{m,n}\Big(\begin{matrix}
		{\bf a} \\  {\bf b}\end{matrix} \Big\vert z; \tau, \alpha\Big)
	=\frac{e^{\frac{\pi\alpha}{\tau}}}{2\pi\i{z}} \int\limits_{\gamma-1} 
	\G_{p,q}^{m,n}\Big(\begin{matrix}{\bf a}\\{\bf b} \end{matrix}
	\Big\vert s+1; \tau\Big)e^{\frac{\pi \alpha}{\tau}{s^2}}\big(e^{-\frac{2\pi\alpha}{\tau}}z\big)^{-s}\d s
	\\
	=\frac{e^{\frac{\pi\alpha}{\tau}}}{2\pi\i{z}} \int\limits_{\gamma-1} F(s)
	\G_{p,q}^{m,n}\Big(\begin{matrix}{\bf a}\\{\bf b} \end{matrix}\Big\vert s; \tau\Big)
	e^{\frac{\pi \alpha}{\tau}{s^2}}\big(e^{-\frac{2\pi\alpha}{\tau}}z\big)^{-s}\d s.
\end{multline*}
The  Mellin convolution theorem \cite[Theorem~44]{Titchmarsh} asserts that
$$
M\Bigg[\int_0^{\infty}f\Big(\frac{x}{t}\Big)K(t)\frac{\d t}{t}\Bigg](s)=M[f](s)M[K](s).
$$
According to \cite[Theorems~2.2 and 3.3]{KilbasSaigo} the  Mellin transform above exists under conditions \eqref{eq:GMellinStrip1} (with the additional restriction $\re(s)<(\re(\mu)-\tau)/N-\tau/2-1$ when $p=q$).
Hence, by the Mellin convolution theorem  we obtain \eqref{eq:firstinteq}. 

The second  quasi-periodicity of Barnes $G$ function in  \eqref{eq:funct_rel_G} leads to the relation
$$
\G_{p,q}^{m,n}(s+\tau)=(2\pi)^{(\tau-1)(p-q)/2}\tau^{\mu-N(\tau+1/2)}\tau^{-sN}\underbrace{\frac{\prod_{j=1}^{n}\Gamma\Big(1+\tau-a_j+s\Big)\prod_{j=n+1}^{p}\Gamma\Big(a_j-\tau-s\Big)}{\prod_{j=1}^{m}\Gamma\Big(b_j-\tau-s\Big)\prod_{j=m+1}^{q}\Gamma\Big(1+\tau-b_j+s\Big)}}_{=H(s)}\G_{p,q}^{m,n}(s),
$$
where $\mu$ and $N$ are defined in \eqref{eq:parameters}. Basic properties of the Mellin transform and  Meijer's $G$ function show that
$\mathcal{M}[h](s)=H(s)$ with $h$ defined in \eqref{eq:phi-defined}. Suppose now that $\lambda_{\flat}-\lambda_{\sharp}>\tau$ and the integral over $\gamma=\L_{\i\infty}$ converges.
Then we can write by changing the integration variable $s\to{s+\tau}$ in \eqref{eq:def_K_function}
\begin{multline*}
	K_{p,q}^{m,n}\Big(\begin{matrix}
		{\bf a} \\  {\bf b}\end{matrix} \Big\vert z; \tau, \alpha\Big)
	=\frac{e^{\pi\alpha\tau}}{2\pi\i{z^{\tau}}}\int\limits_{\gamma-\tau} 
	\G_{p,q}^{m,n}\Big(\begin{matrix}{\bf a}\\{\bf b} \end{matrix}
	\Big\vert s+\tau; \tau\Big)e^{\frac{\pi \alpha}{\tau}{s^2}}\big(e^{-2\pi\alpha}z\big)^{-s}\d s
	\\
	=(2\pi)^{(\tau-1)(p-q)/2}\tau^{\mu-N(\tau+1/2)}\frac{e^{\pi\alpha\tau}}{2\pi\i{z^{\tau}}}\int\limits_{\gamma-\tau} 
	H(s)
	\G_{p,q}^{m,n}\Big(\begin{matrix}{\bf a}\\{\bf b} \end{matrix}
	\Big\vert s; \tau\Big)e^{\frac{\pi \alpha}{\tau}{s^2}}\big(\tau^{N}e^{-2\pi\alpha}z\big)^{-s}\d s.
\end{multline*}
According to \cite[Theorems~2.2 and 3.3]{KilbasSaigo} the  Mellin transform above exists under condition
\eqref{eq:GMellinStrip2} (with the additional restriction $\re(s)<(\re(\mu)-1)/N-\tau-1/2$ when $p=q$).
Hence,  
by the Mellin convolution theorem we obtain \eqref{eq:secondinteq}.
\end{proof}

\textbf{Alternative proof of \eqref{eq:secondinteq}}.  Applying \eqref{eq:transformation_K_z_power_tau}  to $K$ functions on both sides \eqref{eq:firstinteq} we get 
 	\begin{equation*}
	e^{-\frac{\pi\alpha}{\tau}}xK_{p,q}^{m,n}\Big(\begin{matrix}
		\tau^{-1} {\bf a} \\ \tau^{-1} {\bf b}\end{matrix} \Big\vert B x^{\tau}; \tau^{-1}, \tilde \alpha\Big)=
	\int_{0}^{\infty}f\big(e^{-2\pi\alpha/\tau}x/t\big)
	K_{p,q}^{m,n}\Big(\begin{matrix}
		\tau^{-1} {\bf a} \\ \tau^{-1} {\bf b}\end{matrix} \Big\vert e^{-2\pi\alpha}B t^{\tau}; \tau^{-1}, \tilde \alpha\Big)\frac{\d t}{t}.
	\end{equation*}
	Changing the variables $B x^{\tau}=y$ and $Bt^{\tau}= u$ gives us
	 	\begin{equation*}
	e^{-\frac{\pi\alpha}{\tau}} (By)^{1/\tau} K_{p,q}^{m,n}\Big(\begin{matrix}
		\tau^{-1} {\bf a} \\ \tau^{-1} {\bf b}\end{matrix} \Big\vert y; \tau^{-1}, \tilde \alpha\Big)=
	\frac{1}{\tau}\int_{0}^{\infty}f\big(e^{-2\pi\alpha/\tau}  (y/u)^{1/\tau}\big)
	K_{p,q}^{m,n}\Big(\begin{matrix}
		\tau^{-1} {\bf a} \\ \tau^{-1} {\bf b}\end{matrix} \Big\vert \tau^{-N} e^{-2\pi \tilde \alpha}u; \tau^{-1}, \tilde \alpha\Big)\frac{\d u}{u}.
	\end{equation*}
	After rescaling parameters $\tau^{-1} {\mathbf a}\mapsto {\mathbf a}$, 
	$\tau^{-1} {\mathbf b}\mapsto {\mathbf b}$, 
	$\tau^{-1} \mapsto \tau$ and $\tilde \alpha \mapsto \alpha$ we arrive to \eqref{eq:secondinteq}.

\section{Examples of $K$ function in applications}\label{section_examples}

Next we present several examples where $K$-functions appear in applications. All examples in this section will be stated in terms of the function $K_{p,q}^{m,n}\Big( 
\begin{matrix}
{\bf a} \\  {\bf b} 
\end{matrix} \Big
\vert z; \tau
\Big )$ (for which $\alpha=0$), therefore in this section $\alpha$ will be used in different meaning.  

\subsection{Extrema of stable processes}

Let $X=\{X_t\}_{t\ge 0}$ be an $\alpha$-stable L\'evy process started from $X_0=0$, see \cite{kyprianou_pardo_2022}. 
Stable processes are parameterized by a pair $(\alpha,\rho)$, where $\alpha \in (0,2)$ is called the stability parameter (we emphasize again that this $\alpha$ plays different role than $\alpha$ in \eqref{def_phi}). The parameter $\rho:={\mathbb P}(X_1>0)$ is called the positivity parameter and it satisfies the inequalities $\rho \in (0,1)$  and $\alpha-1<\alpha \rho <1$. 
Stable distributions are closely related to gamma functions via Mellin transform. For example, it is well known (see \cite[Theorem 1.13]{kyprianou_pardo_2022})  that 
$$
{\mathbb E}[(X_1)^{z}  {\mathbf 1}_{\{X_1>0\}}]=\frac{1}{\pi} \sin(\pi \rho z) 
\Gamma(z) \Gamma(1-z/\alpha), \;\;\; -1<\re(z)<\alpha. 
$$
Let us define the running supremum of the process $X$ as follows
$$
S_t:=\sup\limits_{0\le s \le t} X_s. 
$$
It was shown in \cite{Kuznetsov2011} that the Mellin transform of the random variable $S_1$  is given by 
\begin{equation}
\e[(S_1)^{z}]=\alpha^{z} \frac{G(\alpha\rho;\alpha)}{G(\alpha(1-\rho)+1;\alpha)}
\times \frac{G(\alpha(1-\rho)+1-z;\alpha)}
{G(\alpha\rho+z;\alpha)}
\times \frac{G(\alpha+z;\alpha)}{G(\alpha-z;\alpha)}, 
\end{equation}
where $-\alpha \rho < \re(z) < \alpha$. 
Writing the probability density function $p_{S_1}(x)$ (of the random variable $S_1$) via the inverse Mellin transform and applying \eqref{eq:transformation_K_shift_by_c} we obtain the following representation (this corresponds to $\tau=\alpha$ and $\alpha=0$ in \eqref{eq:def_K_function})
\begin{equation}
p_{S_1}(x)=
 \frac{G(\alpha\rho;\alpha)}{G(\alpha(1-\rho)+1;\alpha)} x^{-1}
K_{2,2}^{1,1}\Big( 
\begin{matrix}
1, \alpha \\ 1+\alpha(1-\rho), 1+\alpha (1- \rho)
\end{matrix} \Big  \vert x/\alpha; \alpha
\Big ).
\end{equation}
We compute $N=p-q=0$, $\mu=1-\alpha$ and $\nu=\alpha(2\rho-1)-1$. If $\alpha \neq 1$ the parameters correspond to a balanced case and if $\alpha=1$ then necessarily 
$\nu=2\rho-2<0$ and this becomes a linear case. 
When $\alpha$ is irrational, this function can be expanded as a convergent series of powers of $x$, but the convergence of this series is very delicate and depends on subtle number theoretic 
properties of $\alpha$, see \cite{Hackmann2013,HubKuz2011,Kuznetsov2013}. We will summarize some of the results on this series representation here. 
First of all, for any irrational $\alpha$  we  define sequences $\{a_{m,n}\}_{m\ge 0,n\ge 0}$ and  $\{b_{m,n}\}_{m\ge 0,n\ge 1}$ as
follows
\begin{equation}\label{def_a_mn}
a_{m,n}:=\frac{(-1)^{m+n} }{\Gamma\left(1-\rho-n-\frac{m}{\alpha}\right)\Gamma(\alpha\rho+m+\alpha n)}
\prod\limits_{j=1}^{m} \frac{\sin\left(\frac{\pi}{\alpha} \left( \alpha \rho+ j-1 \right)\right)} {\sin\left(\frac{\pi j}{\alpha} \right)} 
\prod\limits_{j=1}^{n} \frac{\sin(\pi \alpha (\rho+j-1))}{\sin(\pi \alpha j)},
\end{equation}
\begin{equation}\label{def_b_mn}
b_{m,n}:=\frac{\Gamma\left(1-\rho-n-\frac{m}{\alpha}\right)\Gamma(\alpha\rho+m+\alpha n) }{\Gamma\left(1+n+\frac{m}{\alpha}\right)\Gamma(-m-\alpha n)}
a_{m,n}.
\end{equation} 
We also introduce a set ${\mathcal A}$, which consists of all of all real irrational numbers $x$, for which there exists a constant $b>1$ such that the inequality
\begin{equation}\label{eqn_def_set_L}
\Big| x -\frac{p}{q} \Big| < \frac{1}{b^{q}}
\end{equation}
is satisfied for infinitely many integers $p$ and $q$. It is known \cite{HubKuz2011}  that this set is dense in $\r$ and has Lebesgue measure zero (and Hausdorff dimension zero).  
The following result was proved in 
\cite{HubKuz2011}. 
\begin{theorem}\label{thm_main}
Assume that $\alpha \notin {\mathcal A} \cup {\mathbb Q}$. Then for all $x>0$ 
\begin{align}\label{eqn_p1}
p_{S_1}(x)& = x^{-1-\alpha } \sum\limits_{n\ge 0} \sum\limits_{m\ge 0}b_{m,n+1} x^{-m-\alpha n}, \;\; {\textnormal { if }} \alpha \in (0,1), \\
\label{eqn_p2}
p_{S_1}(x)&= x^{\alpha\rho-1} \sum\limits_{n\ge 0} \sum\limits_{m\ge 0} a_{m,n} x^{ m+\alpha n}, \;\;\;\;\;\;\;  {\textnormal { if }} \alpha \in (1,2).
\end{align}
The series converges absolutely and uniformly on compact subsets of $(0,\infty)$.  
Moreover, for every $\alpha \notin {\mathbb Q}$ the series \eqref{eqn_p1} \emph{(}resp. \eqref{eqn_p2}\emph{)} provides complete asymptotic expansion as $x\to +\infty$ \emph{(}resp. $x\to 0^+$\emph{)}.
\end{theorem}
It was established in \cite{Kuznetsov2013} that the series \eqref{eqn_p1} diverges for a dense set of irrational numbers $\alpha \in (0,1)$. It is also known that for every irrational $\alpha$ this series can be made to converge conditionally, if one performs summation over triangles $m+1+\alpha(n+1/2) < q_k$ for a carefully chosen sequence $q_k \to +\infty$, see \cite{Hackmann2013} for details.

\subsection{Eigenfunctions of fractional Laplacian on $(0,\infty)$}\label{subsection_fractional_Laplacian}

Let us denote by $F$ the eigenfunction of the fractional Laplace operator 
$(-\Delta)^{\alpha/2}$ on $(0,\infty)$ with Dirichlet condition on 
$(-\infty,0]$. This problem is related to an $\alpha$-stable process killed at the first exit from $(0,\infty)$, see \cite{Kuznetsov2018}.  It was first established in 
\cite{Kwasnicki2011} that $F$ can be expressed as follows 
$$
F(x)=\sin(x+\pi(2-\alpha)/8)-g(x), 
$$
for certain completely monotone bounded function $g$.  The Mellin transform of $F$  can be obtained from Theorem 1.5 in \cite{Kuznetsov2018} by setting $\rho=1/2$:
\begin{equation}\label{Mellin_F_S2}
\int_0^{\infty} F(x)x^{z-1} \d x=\frac{ \Gamma(z)  S_2(z;\alpha)}{2 S_2(\alpha/2 +z;\alpha)}, \;\;\;  -\alpha/2< \re(z)<0.  
\end{equation}
Here $S_2$ is the double sine function \cite{Barnes1899,Kuznetsov2018}:
\begin{equation}\label{eqn_S2_Barnes_Gamma}
S_2(z;\alpha):=(2\pi)^{(1+\alpha)/2-z} \frac{G(z;\alpha)}{G(1+\alpha-z;\alpha)}. 
\end{equation}
This function satisfies the two functional equations 
\begin{equation}\label{S2_two_functional_eqns}
S_2(z+1;\alpha)=\frac{S_2(z;\alpha)}{2\sin(\pi z/\alpha)}, \;\;\; S_2(z+\alpha;\alpha)=\frac{S_2(z;\alpha)}{2\sin(\pi z)},
\end{equation}
and the normalizing condition $S_2((1+\alpha)/2;\alpha)=1$. 
Expressing the double sine functions in \eqref{Mellin_F_S2} in terms of double gamma functions, we obtain 
\begin{align*}
\int_0^{\infty} F(x)x^{z-1} \d x&=\frac{1}{2} (2\pi)^{\alpha/2} \Gamma(z) 
\frac{G(z;\alpha) G(1+\alpha/2-z;\alpha)}
{G(1+\alpha-z;\alpha) G(\alpha/2+z;\alpha)}
\\
&=
\alpha^z \sqrt{\frac{\pi}{2\alpha}}
\frac{G(z+\alpha;\alpha)G(1+\alpha/2-z;\alpha)}
{G(1+\alpha-z;\alpha) G(\alpha/2+z;\alpha)}, 
\end{align*}
where in the last step we used the second functional equation from \eqref{eq:funct_rel_G}. 
From the above formula we obtain the desired representation of the eigenfunction of fractional Laplacian
\begin{equation}
F(x)=
\sqrt{\frac{\pi \alpha}{2}} x^{-1}
K_{2,2}^{1,1}\Big( 
\begin{matrix}
0, \alpha \\ \alpha/2, \alpha/2
\end{matrix} \Big  \vert x/\alpha; \alpha
\Big ).
\end{equation}
In this case $N=p-q=0$ and $\mu=-\alpha$, which corresponds to a balanced case. 

\bigskip

\subsection{Exponential functionals of hypergeometric processes}

Let $X=\{X_t\}_{t\ge 0}$ be a hypergeometric L\'evy process, defined via the Laplace exponent 
$$
\psi(z)=\ln \e [\exp(z X_1)]=-\frac{\Gamma(1-\beta+\gamma-z)
\Gamma(\hat \beta+\hat \gamma+z)}{\Gamma(1-\beta-z)\Gamma(\hat \beta+z)}.
$$
Here the parameters satisfy $\beta \le 1$, $\hat \beta\ge 0$ and $\gamma,\hat \gamma \in (0,1)$. 
Such processes were studied \cite{KuznetsovPardo2013}, where it was shown that they are generalizations of the Lamperti-stable processes \cite{kyprianou_pardo_2022}. A very important random variable that is used in the study of $\alpha$-stable processes
is the exponential functional of hypergeometric processes, defined as 
$$
I=I(\alpha,X)=\int_0^{\zeta-} e^{-\alpha X_t} \d t,
$$
where $\alpha>0$ and $\zeta$  is the lifetime of the process $X$. 
The Mellin transform of $I$ is defined as $M(s):=\e[I^{s-1}]$ and it is known \cite{KuznetsovPardo2013} that $M$ satisfies the equation 
$M(s+1)=-M(s)/\psi(\alpha s)$. Given that $\psi$ is a product of gamma functions, this equation  can be easily solved in terms of the double gamma functions (via \eqref{eq:funct_rel_G}) and a certain uniqueness argument  \cite[Proposition 2]{KuznetsovPardo2013} then shows that  
\begin{equation}\label{exp_functional}
M(s)=C \Gamma(s) \frac{G((1-\beta)\delta+s;\delta)
G((\hat \beta+\hat \gamma)\delta+1-s;\delta)}
{G((1-\beta+\gamma)\delta+s;\delta) G(\hat \beta \delta+1-s;\delta)},
\end{equation}
where $\delta:=1/\alpha$ and $C$ is a normalizing constant which makes the expression in the right-hand side equal to $1$ when $s=1$. 
Expressing the gamma function in \eqref{exp_functional} as a ratio of two double gamma functions via \eqref{gamma_as_G} and writing the probability density function $p_I(x)$ (of random variable $I$) as the inverse Mellin transform, we obtain the following expression 
\begin{equation}
p_I(x)=
C\sqrt{\delta} (2\pi)^{(1-\delta)/2}
x^{-1} K_{3,3}^{1,2}\Big( 
\begin{matrix}
0, \beta\delta, \hat \beta \delta \\
(\hat \beta+\hat \gamma)\delta,  (\beta-\gamma)\delta, \delta
\end{matrix} \Big  \vert x/\delta; \delta
\Big ).
\end{equation}

In deriving the above formula we also used \eqref{eq:transformation_K_shift_by_c}. 
Some results on series representation of $p_I(x)$ can be found in \cite[Section 4]{KuznetsovPardo2013}. In this case we have $N=p-q=0$, 
$\mu=\delta(\gamma+\hat \gamma-1)$ and $\nu=\delta(\gamma-\hat \gamma-1)$, thus this example falls either in the balanced case (when $\gamma+\hat\gamma\neq 1$) or the linear case (when 
$\gamma+\hat\gamma=1$).

\subsection{Generalized Kilbas-Saigo function}\label{section_Kilbas_Saigo}

Following \cite{BS2021} we introduce a Pochhammer type symbol 
$$
[a;\tau]_z:=\frac{G(a+z;\tau)}{G(a;\tau)}.
$$
Note that in view of the first relation in \eqref{eq:funct_rel_G} for $n=0,1,2,\dots$ we have
$$
[a;\tau]_{0}=1,~~[a;\tau]_n
=\prod\limits_{k=0}^{n-1}
\Gamma((a+k)/\tau). 
$$
Assume that $\tau>0$, $a_j \in \r \setminus \{ -m\tau-n  \vert m,n\ge 0\}$, $b_j>0$ for $j=1,\dots,p$,  and 
\begin{equation}\label{assumption_KS}
 \sum\limits_{j=1}^p (b_j-a_j)>0. 
\end{equation}
Define the function 
\begin{equation}\label{def_gen_KS_function}
f(x)=
\sum\limits_{n\ge 0} \frac{[a_1,\dots,a_p;\tau]_n}{[b_1,\dots,b_p;\tau]_n}x ^n,
\end{equation}
where 
$$
[a_1,\dots,a_p;\tau]_n:=\prod\limits_{j=1}^p 
[a_j;\tau]_n
$$
From Proposition \ref{prop_asymptotics_G} we find that
\begin{equation}\label{asymptotics_Pocchammer}
 \frac{[a;\tau]_z}{[b;\tau]_z}=
  \frac{G(b;\tau)}{G(a;\tau)} \exp\Big( 
  \frac{a-b}{\tau} z\ln(z)
  -\frac{(a-b)}{\tau}\Big(1+\ln(\tau)\Big) z
  +O(\ln |z|) \Big), \;\;\; z\to \infty,
\end{equation}
which holds uniformly in any sector $|\arg(z)|<\pi - \epsilon. $
The above asymptotic result and our assumption 
\eqref{assumption_KS} guarantee that $f(x)$ is an entire function. This function is a natural generalization of the Kilbas-Saigo function (see 
 \cite[Chapter~5.2]{GKMR2020} and \cite{BS2021}), defined as
$$
E_{\alpha,m,l}(x)=1+\sum\limits_{n\ge 1} \Big[ \prod\limits_{k=0}^{n-1} 
\frac{\Gamma(1+\alpha(km+l))}{\Gamma(1+\alpha(km+l+1))} \Big]x^n=
\sum\limits_{n\ge 0} \frac{[\tilde a;\tilde \tau]_n}{[\tilde b;\tilde \tau]_n} x^n,
$$
with $\tilde \tau=1/(\alpha m)$, $\tilde{a}=(\alpha l+1) \tilde \tau$ and $\tilde{b}=\tilde{a}+1/m$. 

As we show in the next result, the function 
$f$ defined in \eqref{def_gen_KS_function} is a special case of the Meijer-Barnes $K$-function.

\begin{proposition}\label{prop_KS_as_MB}
Assume that  $0<\sum\limits_{j=1}^p (b_j-a_j)<2 \tau$. Then for $x>0$  
\begin{equation}\label{KS_as_MB}
f(-x)= C \times K^{p+1,1}_{p+2,p+2}\Big(\begin{matrix}
		1,1,b_1,\dots,b_p \\a_1,\dots,a_p,1+\tau,1+\tau\end{matrix} \Big\vert x; \tau\Big),
\end{equation}
where 
$$
C:=(2\pi)^{1-\tau} 
\prod\limits_{j=1}^p \frac{G(b_j;\tau)}{G(a_j;\tau)}.   
$$
\end{proposition}
\begin{proof}
Let us denote 
$$
\varphi(z):=\frac{[a_1,\dots,a_p;\tau]_z}{[b_1,\dots,b_p;\tau]_z}.
$$
We have 
$$
f(-x)=\sum\limits_{n\ge 0} \varphi(n) (-x)^n
$$
and using \eqref{asymptotics_Pocchammer}  one can check that function $\varphi$ satisfies the conditions of Ramanujan's Master Theorem (see \cite{Amdeberhan2012}), thus formula (3.4) in \cite{Amdeberhan2012} gives us
\begin{equation}\label{eq57}
f(-x)=\frac{1}{2\pi \i} \int_{c+\i \r} 
\frac{\pi}{\sin(\pi s)} \varphi(-s) x^{-s} \d s,
\end{equation}
where $0<c < \min_{1\le{j}\le{p}} (b_j)$. In view of \eqref{gamma_as_G}
 we have
\begin{equation}\label{eq58}
\frac{\pi}{\sin(\pi{s})}=\Gamma(s)\Gamma(1-s)=(2\pi)^{1-\tau} \frac{G(s+\tau;\tau)G(1-s+\tau;\tau)}{G(s;\tau)G(1-s;\tau)}.
\end{equation}
Given the asymptotics in \eqref{asymptotics_Pocchammer}, we can change the contour of integration in \eqref{eq57} to ${\mathcal L}_{-\pi + \epsilon, \pi-\epsilon}$ and now,
 using \eqref{eq58} and comparing \eqref{eq57} with our definition of Meijer-Barnes $K$-function, we obtain the desired result 
 \eqref{KS_as_MB}. 
\end{proof}

Using exactly the same method one can establish the following result. 

\begin{proposition}\label{prop_KS_as_MB2}
The function
\begin{equation}\label{def_tilde_f}
    \tilde f(x)=\sum\limits_{n\ge 0} \frac{[a_1,\dots,a_p;\tau]_n}{[b_1,\dots,b_p;\tau]_n} \frac{x ^n}{n!},
\end{equation}
is entire if $0<\sum\limits_{j=1}^p (b_j-a_j)$. If, in addition, we have $\sum\limits_{j=1}^p (b_j-a_j)<\tau$ then 
\begin{equation}\label{KS_as_MB2}
f(-x)= C \times K^{p,1}_{p+1,p+1}\Big(\begin{matrix}
		1,b_1,\dots,b_p \\a_1,\dots,a_p,1+\tau\end{matrix} \Big\vert x/\tau; \tau\Big),  \;\;\; x>0,
\end{equation}
where 
$$
C:=(2\pi)^{(1-\tau)/2} \tau^{-1/2} 
\prod\limits_{j=1}^p \frac{G(b_j;\tau)}{G(a_j;\tau)}.  
$$
\end{proposition}
One can check that both of the examples in this section belong to the balanced case of parameters. The functions of the form \eqref{def_tilde_f} appeared in \cite{Ostrovsky} as the Laplace transforms of Barnes beta distributions, to which we now turn our attention.

\subsection{Barnes beta distributions}\label{section_beta_distributions}

Ostrovsky \cite{Ostrovsky} introduced a family of higher-order beta distributions, called Barnes beta distributions (similar distributions appeared later in \cite{BS2021}). A particularly interesting example in that family is the random variable $\beta_{2,2}(\tau,{\mathbf w})$, defined via Mellin transform 
$$
\e[\beta_{2,2}(\tau,{\mathbf w})^s]=
\frac{\Gamma_2(s+w_0;1,\tau) \Gamma_2(w_0+w_1;1,\tau) \Gamma_2(w_0+w_2;1,\tau) \Gamma_2(s+w_0+w_1+w_2;1,\tau)}
{\Gamma_2(w_0;1,\tau) \Gamma_2(s+w_0+w_1;1,\tau) \Gamma_2(s+w_0+w_2;1,\tau) \Gamma_2(w_0+w_1+w_2;1,\tau)}. 
$$
Here $\re(s)>-w_0$ and ${\mathbf w}=(w_0,w_1,w_2)$ with $w_i>0$.  It is known that the random variable $\beta_{2,2}(\tau,{\mathbf w})$ is supported on $(0,1)$, it has absolutely continuous distribution  and that its logarithm is infinitely divisible  (see Theorem 2.4 and Corollary 2.5 in \cite{Ostrovsky}).  The random variable $\beta_{2,2}(\tau,{\mathbf w})$ has interesting connection to the celebrated Selberg's integral (see also \cite{Ostrovsky_2012}). It turns out that a product of five independent random variables, three of which are equal in distribution to $1/\beta_{2,2}(\tau,{\mathbf w})$ (with different parameters ${\mathbf w}$) and the remaining two have Frechet and lognormal distributions, has moments given by the Selberg's integral (see Theorems 4.1 and 4.2 in \cite{Ostrovsky}). 

Using \eqref{G_as_Gamma1} we can rewrite the above Mellin transform in the form
$$
\e[\beta_{2,2}(\tau,{\mathbf w})^s]=
\frac{G(w_0;\tau) G(s+w_0+w_1;\tau) G(s+w_0+w_2;\tau) G(w_0+w_1+w_2;\tau)}{G(s+w_0;\tau) G(w_0+w_1;\tau) G(w_0+w_2;\tau) G(s+w_0+w_1+w_2;\tau)}. 
$$

Writing the density $q(x)$ of the random variable $\beta_{2,2}(\tau,{\mathbf w})$ as the inverse Mellin transform and applying
\eqref{eq:transformation_K_shift_by_c} and \eqref{eq:tranformation_K_z_1/z} we obtain the following result: for $w_1 w_2>\tau$ and $x>0$
$$
q(x)=C x^{w_0-1}
K_{2,2}^{2,0}\Big( 
\begin{matrix}
0, w_1+w_2 \\
w_1, w_2
\end{matrix} \Big  \vert x^{-1}; \tau
\Big ),
$$
where 
$$
C=\frac{G(w_0;\tau)  G(w_0+w_1+w_2;\tau)}{G(w_0+w_1;\tau) G(w_0+w_2;\tau) }.
$$
This is an example of the logarithmic case, as here we have $N=p-q=\mu=\nu=0$ and $\xi=-2w_1w_2$. 
Theorem~\ref{thm_logarithmic_case} tells us that $q(x)=0$ for $x>1$, as previously stated in \cite{Ostrovsky}.



\appendix

\setcounter{equation}{0}
\renewcommand{\theequation}{\Alph{section}.\arabic{equation}}

\section{Alternative definition in terms of 
$\Gamma_2(z;\omega_1,\omega_2)$}\label{AppendixA}
  
Here we will follow \cite{Spreafico2009}. 
Let $\omega_1$ and $\omega_2$ be positive numbers and denote $\boldsymbol{\omega}:=(\omega_1, \omega_2)$. The double gamma function $\Gamma_2(z;\boldsymbol{\omega})$ ($\Gamma^B(z;\boldsymbol{\omega})$ in the notation  of \cite{Ruijsenaars2000}) as defined in \cite[(3), Corollary~5.3, Proposition~8.3]{Spreafico2009} is related to the $G$ function via \eqref{G_as_Gamma1}.  The $\Gamma_2$ function satisfies the functional equations 
\eqref{Gamma_2_functional_equations}, 
as well as the identity 
\begin{equation*}
\Gamma_2(cz;{c}\omega_1,{c}\omega_2)=c^{(\frac{1}{\omega_1}+\frac{1}{\omega_2})\frac{z}{2}-\frac{z^2}{2\omega_1\omega_2}-1}
\Gamma_2(z;\omega_1,\omega_2), \;\;\; c>0,
\end{equation*}
from which one can derive formula \eqref{G_as_Gamma1}.

We define  $\tilde K$-function as 
\begin{equation}\label{def_K_function2}
\tilde K_{p,q}^{m,n}\Big( 
\begin{matrix}
{\mathbf a} \\ {\mathbf b} \end{matrix} \Big  \vert z;\boldsymbol{\omega}, \alpha \Big )
=\frac{1}{2\pi \i} \int\limits_{\tilde\gamma} 
\frac{\prod\limits_{j=1}^m \Gamma_2(b_j-s;\boldsymbol{\omega})
	 \prod\limits_{j=1}^n \Gamma_2(\omega_1+\omega_2-a_j+s;\boldsymbol{\omega})}
	{\prod\limits_{j=m+1}^q \Gamma_2(\omega_1+
 \omega_2-b_j+s;\boldsymbol{\omega}) 
	 \prod\limits_{j=n+1}^p \Gamma_2(a_j-s;\boldsymbol{\omega})} e^{\frac{\pi \alpha}{
     \omega_1 \omega_2} s^2} z^{-s}ds,
\end{equation}
where the contour $\tilde\gamma=\omega_1\gamma$ with $\gamma$ chosen according to Definition~\ref{def_K_function} applied to the $K$-function on the right hand side of the following relation:
\begin{equation}\label{eq:Ktilde-K}
\tilde K_{p,q}^{m,n}\Big( 
\begin{matrix}
a_1, \dots, a_p \\ b_1, \dots, b_q
\end{matrix} \Big  \vert z; \boldsymbol{\omega},\alpha\Big ) 
=C_1 K_{q,p}^{p-n,q-m}\Bigg( 
\begin{matrix}
\tilde a_1, \dots, \tilde a_q \\ \tilde b_1, \dots, \tilde b_p
\end{matrix} \Big  \vert C_2 z^{\omega_1}; \tau,\eta \Bigg),
\end{equation}
which we will prove below.  Here $\tau=\omega_2/\omega_1$,  $\eta=\alpha-N\ln(\omega_2)/(2\pi)$,  $\tilde a_j=b_{q-j+1}/\omega_1$, $\tilde b_j=a_{p-j+1}/\omega_1$  and 
$$
C_1=(2\pi)^{((\omega_1+\omega_2)(m+n-q)-\nu)/(2\omega_1)}\omega_2^{-N-\xi/(2\omega_1\omega_2)+\mu(\omega_1+\omega_2)/(2\omega_1\omega_2)}\omega_1,
$$
$$
C_2=(2\pi)^{(q-p)/2}\omega_2^{(N(\omega_1+\omega_2)/2-\mu)/\tau},
$$
with $N$, $\nu$, $\mu$  and $\xi$  defined in \eqref{eq:parameters}.
Indeed, in view of \eqref{G_as_Gamma1} we obtain:
\begin{multline*}
    \frac{\prod\limits_{j=1}^m \Gamma_2(b_j-s;\boldsymbol{\omega})
	 \prod\limits_{j=1}^n \Gamma_2(\omega_1+\omega_2-a_j+s;\boldsymbol{\omega})}
	{\prod\limits_{j=m+1}^q \Gamma_2(\omega_1+\omega_2-b_j+s;\boldsymbol{\omega}) 
	 \prod\limits_{j=n+1}^p \Gamma_2(a_j-s;\boldsymbol{\omega})}
	 \\
	 =(2\pi)^{((p-q)s+(\omega_1+\omega_2)(m+n-q)-\nu)/(2\omega_1)}\omega_2^{(-Ns^2+2\mu{s}-\xi)/(2\omega_1\omega_2)-N}\omega_2^{(\omega_1+\omega_2)(-Ns+\mu)/(2\omega_2\omega_2)}
	\\
	\times\frac{\prod_{j=m+1}^{q}G\big(1+\tau-b_j/\omega_1+s/\omega_1;\tau\big)\prod_{j=n+1}^{p}G\big(a_j/\omega_1-s/\omega_1;\tau\big)}
	{\prod_{j=1}^{m}G\big(b_j/\omega_1-s/\omega_1;\tau\big)\prod_{j=1}^{n}G\big(1+\tau-a_j/\omega_1+s/\omega_1;\tau\big)}.
\end{multline*}
Comparing with the  definition \eqref{eq:def_K_function}
$$
K_{p,q}^{m,n}\Big( 
\begin{matrix}
{\bf a} \\  {\bf b} 
\end{matrix} \Big
\vert z; \tau, \eta
\Big )=\frac{1}{2\pi \i} \int\limits_{\gamma} 
\frac{\prod\limits_{j=1}^m G(b_j-s;\tau)\prod\limits_{j=1}^n G(1+\tau-a_j+s;\tau)}
	{\prod\limits_{j=m+1}^q G(1+\tau-b_j+s;\tau)\prod\limits_{j=n+1}^p G(a_j-s;\tau)}e^{\frac{\pi \eta}{\tau}{s^2}}z^{-s}ds,
$$
we obtain:
\begin{multline*}
    \tilde K_{p,q}^{m,n}\Big( 
\begin{matrix}
{\mathbf a} \\ {\mathbf b}
\end{matrix} \Big  \vert z;\boldsymbol{\omega}
\Big )=\frac{(2\pi)^{((\omega_1+\omega_2)(m+n-q)-\nu)/(2\omega_1)}}{2\pi \i}\omega_2^{-N-\xi/(2\omega_1\omega_2)+\mu(\omega_1+\omega_2)/(2\omega_1\omega_2)}
\\
\times\int\limits_{\tilde\gamma}\frac{\prod_{j=m+1}^{q}G\big(1+\tau-b_j/\omega_1+s/\omega_1;\tau\big)\prod_{j=n+1}^{p}G\big(a_j/\omega_1-s/\omega_1;\tau\big)}
	{\prod_{j=1}^{m}G\big(b_j/\omega_1-s/\omega_1;\tau\big)\prod_{j=1}^{n}G\big(1+\tau-a_j/\omega_1+s/\omega_1;\tau\big)} 
	\\
	\times (2\pi)^{(p-q)s/(2\omega_1)}\omega_2^{(-Ns^2+2\mu{s})/(2\omega_1\omega_2)}\omega_2^{-Ns(\omega_1+\omega_2)/(2\omega_1\omega_2)}e^{\frac{\pi \alpha}{
     \omega_1 \omega_2} s^2}z^{-s}ds
\\
=\frac{(2\pi)^{((\omega_1+\omega_2)(m+n-q)-\nu)/(2\omega_1)}}{2\pi \i}\omega_2^{-N-\xi/(2\omega_1\omega_2)+\mu(\omega_1+\omega_2)/(2\omega_1\omega_2)}\omega_1 
\\
\times\int\limits_{\gamma}\frac{\prod_{j=m+1}^{q}G\big(1+\tau-b_j/\omega_1+s;\tau\big)\prod_{j=n+1}^{p}G\big(a_j/\omega_1-s;\tau\big)}
	{\prod_{j=1}^{m}G\big(b_j/\omega_1-s;\tau\big)\prod_{j=1}^{n}G\big(1+\tau-a_j/\omega_1+s;\tau\big)} 
	e^{\pi\eta{s^2}/\tau}(C_2z^{\omega_1})^{-s}ds
\\
=(2\pi)^{((\omega_1+\omega_2)(m+n-q)-\nu)/(2\omega_1)}\omega_2^{-N-\xi/(2\omega_1\omega_2)+\mu(\omega_1+\omega_2)/(2\omega_1\omega_2)}\omega_1 K^{p-n,q-m}_{q,p}\Bigg( 
\begin{matrix}
   \tilde {\bf a}\\ 
   \tilde {\bf b} 
\end{matrix} \Big
\vert C_2 z^{\omega_1}; \tau, \eta\Bigg),	
\end{multline*}
thus proving \eqref{eq:Ktilde-K}.

\setcounter{equation}{0}
\renewcommand{\theequation}{\Alph{section}.\arabic{equation}}
\section{Proof of \eqref{eq:G_s_over_tau_transformation}}\label{AppendixB}

We denote $C:=(2\pi)^{\frac{1}2 \left(1-\frac1{\tau} \right)} 
\tau^{\frac{1}{2}(1+\frac{1}{\tau})}$ and recall that $N$, $\mu$, $\nu$ and $\xi$ are defined in  \eqref{eq:parameters}.  From \eqref{eq:G_1_over_tau} we have 
\begin{equation}\label{eq:G_1_over_tau_v2}
G(z;\tau)=
C^z \tau^{-\frac{z^2}{2\tau}-1} 
 G\left(\frac{z}{\tau};\frac{1}{\tau}\right).
\end{equation}
We apply the above formula to every double gamma factor in \eqref{eq:integrand} and obtain
\begin{align}\label{G_transformation_proof1}
\G_{p,q}^{m,n}\Big(\begin{matrix} {\bf a} \\
{\bf b} \end{matrix}
\Big\vert s; \tau\Big)&=
C^{D_1} 
\tau^{-N-D_2/(2\tau)}\
\frac{\prod\limits_{j=1}^m G(b_j/\tau-s/\tau;1/\tau)\prod\limits_{j=1}^n G(1+1/\tau-a_j/\tau+s/\tau;1/\tau)}
	{\prod\limits_{j=m+1}^q G(1+1/\tau-b_j/\tau+s/\tau;1/\tau)\prod\limits_{j=n+1}^p G(a_j/\tau-s/\tau;1/\tau)}\\
 &=C^{D_1} \tau^{-N-D_2/(2\tau)} \G_{p,q}^{m,n}\Big(\begin{matrix} \tau^{-1} {\bf a} \\ \nonumber
\tau^{-1} {\bf b} \end{matrix}
\Big\vert s\tau^{-1}; \tau^{-1}\Big),
\end{align}
where we denoted for $r=1,2$,
\begin{align*}
    D_r=\sum_{j=1}^m (b_j-s)^r+\sum_{j=1}^n (1+\tau-a_j+s)^r
-\sum_{j=m+1}^q (1+\tau-b_j+s)^r-\sum_{j=n+1}^p (a_j-s)^r.
\end{align*}
Now our goal is to simplify the expressions for $D_1$ and $D_2$. The first one is easy: it is straightforward to  check that 
\begin{align*}
D_1=s(p-q)+(1+\tau)(n+m-q)-\nu.
\end{align*}
Simplifying an expression for $D_2$ takes more work:
\begin{align*}
D_2&=\sum_{j=1}^m (b_j-s)^2+\sum_{j=1}^n (1+\tau-a_j+s)^2
-\sum_{j=m+1}^q (1+\tau-b_j+s)^2-\sum_{j=n+1}^p (a_j-s)^2\\
&=\sum_{j=1}^n a_j^2-\sum_{j=n+1}^p a_j^2+\sum_{j=1}^m b_j^2-
\sum_{j=m+1}^q b_j^2+Ns^2+(1+\tau)^2(n+m-q) \\
&+2s\bigg(\sum_{j=1}^n (1+\tau-a_j)+ \sum_{j=n+1}^p a_j
-\sum_{j=1}^m b_j - \sum_{j=m+1}^q (1+\tau-b_j)\bigg)\\
&-2(1+\tau) \Big( \sum_{j=1}^n a_j-\sum_{j=m+1}^q b_j\Big)\\
&=\sum_{j=1}^n a_j^2-\sum_{j=n+1}^p a_j^2+\sum_{j=1}^m b_j^2-
\sum_{j=m+1}^q b_j^2+Ns^2+(1+\tau)^2(n+m-q) \\
&+2s\Big((1+\tau)(n+m-q)-\sum_{j=1}^n a_j+ \sum_{j=n+1}^p a_j
-\sum_{j=1}^m b_j + \sum_{j=m+1}^q b_j\Big)\\
&-2(1+\tau) \Big( \sum_{j=1}^n a_j-\sum_{j=m+1}^q b_j\Big)\\
&=\xi+Ns^2+(1+\tau)^2(n+m-q)+2s\big((1+\tau)(n+m-q)-\mu\big)-(1+\tau)(\mu+\nu).
\end{align*}
Thus we have 
\begin{align*}
C^{D_1} 
\tau^{-N-D_2/(2\tau)}&=C^{s(p-q)+(1+\tau)(n+m-q)-\nu} \tau^{-N-\frac{1}{2\tau} \big(\xi+Ns^2+(1+\tau)^2(n+m-q)+2s((1+\tau)(n+m-q)-\mu)-(1+\tau)(\mu+\nu)\big)}\\
&=e^{-s^2 N\frac{\ln(\tau)}{2\tau}} \times \Big(C^{q-p} \tau^{\frac{1}{\tau}((1+\tau)(n+m-q)-\mu)} \Big)^{-s} \\
&\times C^{(1+\tau)(n+m-q)-\nu}
\tau^{-N-\frac{1}{2\tau}\big(\xi+(1+\tau)^2(n+m-q)-(1+\tau)(\mu+\nu)\big)}.
\end{align*}
Recall that $C=(2\pi)^{\frac{\tau-1}{2\tau}} 
\tau^{\frac{1+\tau}{2 \tau }}$, 
thus 
$$
C^{q-p} \tau^{\frac{1}{\tau}((1+\tau)(n+m-q)-\mu)}
=(2\pi)^{\frac{1-\tau}{2\tau}(p-q)} \tau^{\frac{1}{2\tau}(N(1+\tau)-2\mu)},
$$
and 
\begin{align*}
&C^{(1+\tau)(n+m-q)-\nu}
\tau^{-N-\frac{1}{2\tau} \big(\xi+(1+\tau)^2(n+m-q)-(1+\tau)(\mu+\nu)\big)}\\
&=(2\pi)^{\frac{\tau-1}{2\tau}((1+\tau)(n+m-q)-\nu)}
\times \tau^{\frac{1}{2\tau} ((1+\tau)((1+\tau)(n+m-q)-\nu)-\xi-(1+\tau)^2(n+m-q)+(1+\tau)(\mu+\nu)-N}\\
&=(2\pi)^{\frac{\tau-1}{2\tau}((1+\tau)(n+m-q)-\nu)}
\times \tau^{\frac{1}{2\tau} ((1+\tau)\mu-\xi)-N}.
\end{align*}
Combining \eqref{G_transformation_proof1} and the above three formulas we 
 finally obtain the desired result \eqref{eq:G_s_over_tau_transformation}. 
 
 \vspace{1.0cm}


\end{document}